\pdfoutput=1
\documentclass[11pt]{article}

\usepackage{xspace}
\newcommand{\nips}[1]{foo}
\newcommand{\notnips}[1]{bar}
\ifcsname NipsWokshopFTW\endcsname%
\renewcommand{\nips}[1]{#1}%
\renewcommand{\notnips}[1]{}%
\else%
\renewcommand{\nips}[1]{}%
\renewcommand{\notnips}[1]{#1}%
\fi%

\usepackage{enumerate}
\usepackage{fancyhdr,amsmath,amssymb,amsthm,url}

\usepackage{accents}

\usepackage{booktabs}
\usepackage{float}
\usepackage{multirow}
\usepackage[normalem]{ulem}
\usepackage{accents}
\notnips{\usepackage{graphicx}}				
\nips{\usepackage{graphicx}}
\usepackage{amssymb}
\usepackage{color}
\usepackage{xspace}
\usepackage[numbers,square]{natbib}
\usepackage{array}

\usepackage{mathtools}
\usepackage{thmtools}
\usepackage{thm-restate}

\usepackage{algorithm}
\usepackage[noend]{algpseudocode}

\usepackage{etoolbox}

\notnips{\usepackage[top=1in, right=1in, left=1in, bottom=1in]{geometry}}
\notnips{\usepackage[hidelinks]{hyperref}}
\nips{\usepackage[hidelinks]{hyperref}}

\usepackage{placeins}

\usepackage{tikz}

\ifcsname NipsWokshopFTW\endcsname%
\else%
\makeatletter
\long\def\@makecaption#1#2{
  \vskip 0.8ex
  \setbox\@tempboxa\hbox{\small {\bf #1:} #2}
  \parindent 1.5em  
  \dimen0=\hsize
  \advance\dimen0 by -3em
  \ifdim \wd\@tempboxa >\dimen0
  \hbox to \hsize{
    \parindent 0em
    \hfil 
    \parbox{\dimen0}{\def\baselinestretch{0.96}\small
      {\bf #1.} #2
    } 
    \hfil}
  \else \hbox to \hsize{\hfil \box\@tempboxa \hfil}
  \fi
}
\makeatother
\fi%

\usepackage[noend]{algpseudocode}
\usepackage{xfrac}
\usepackage[nice]{nicefrac}

\input{macros.sty} 

\usepackage{xr}
\def\usenewertable{1}

\title{Lower Bounds for Finding Stationary Points II:\\First-Order Methods}

\author{Yair Carmon ~~~ John C.\ Duchi ~~~ Oliver Hinder ~~~ Aaron Sidford \\
  \texttt{\{\href{mailto:yairc@stanford.edu}{yairc},
    \href{mailto:jduchi@stanford.edu}{jduchi},
    \href{mailto:ohinder@stanford.edu}{ohinder},
    \href{mailto:sidford@stanford.edu}{sidford}\}@stanford.edu}}
\date{}

\providecommand{\compactfunc}{\Psi}

\begin{document}

\maketitle

\begin{abstract}
 We establish lower bounds on the complexity of finding
 $\epsilon$-stationary points of smooth, non-convex
 high-dimensional functions using first-order methods. We prove 
 that deterministic first-order methods, even applied to arbitrarily smooth
 functions, cannot achieve convergence rates in $\epsilon$ better than
 $\epsilon^{-8/5}$, which is within $\epsilon^{-1/15}\log\frac{1}{\epsilon}$ of 
 the best known rate for such methods. Moreover, for functions with Lipschitz 
 first and second
 derivatives, we prove no deterministic first-order method can achieve 
 convergence rates better than
 $\epsilon^{-12/7}$, while $\epsilon^{-2}$ is a lower bound for
 functions with only Lipschitz gradient. For \emph{convex} functions with 
 Lipschitz gradient, accelerated gradient descent achieves the rate 
 $\epsilon^{-1}\log\frac{1}{\epsilon}$, showing that finding stationary points 
 is  easier given convexity.
\end{abstract}

%
\section{Introduction}

We study the oracle complexity of finding approximate
stationary points of a smooth function $f : \R^d \to \R$, that is,
a point $x$ such that
\begin{equation}
  \norm{\nabla f(x)} \le \epsilon.
  \label{eqn:stationarity}
\end{equation}
In Part I of this series~\cite{NclbPartI}, we establish the complexity
of finding an $\epsilon$-stationary point~\eqref{eqn:stationarity} for
algorithms that, at a query point $x$, have access
to all derivatives of $f$. In contrast, in this paper we focus on
\emph{first-order methods}, which only query function values and gradients.

First-order methods are important in large-scale optimization for many
reasons. Perhaps the two most salient are that each iteration is often
inexpensive, and that on many problems, the number of iterations grows 
slowly (or not at all) 
with the problem dimension $d$. %
 From a theoretical
perspective, the latter property is captured by %
\emph{dimension-free} convergence rates, %
where the worst case iteration count
depends polynomially on the desired accuracy and measures of function
regularity but
has no explicit dependence on $d$. 
In non-convex optimization problems, regularity often comes
by assuming bounded function value at the initial point $x\ind{0}$, i.e.\
$f(x\ind{0}) - \inf_x f(x) \le \DeltaF$ for some $\DeltaF >0$, and that
$\grad f$ is $\SmGrad$-Lipschitz continuous. Under these 
conditions, classical gradient descent finds an
$\epsilon$-stationary point in $2\SmGrad\DeltaF
\epsilon^{-2}$ iterations~\cite{Nesterov04}, a dimension-free guarantee.

Developing first-order methods for finding stationary points of non-convex
functions with improved dimension-free rates of convergence is an area of
active research~\cite{CarmonDuHiSi16, AgarwalAlBuHaMa17, CarmonDuHiSi17,
  Allen17b}.  Under the additional assumption of Lipschitz second
derivatives, we~\cite{CarmonDuHiSi16} and \citet{AgarwalAlBuHaMa17} propose
randomized first-order methods with nearly dimension free rate
$\epsilon^{-7/4} \log\frac{d}{\epsilon}$ (ignoring other problem-dependant
constants). In a later paper~\cite{CarmonDuHiSi17}, we propose a
deterministic accelerated gradient-based method with complexity
$\epsilon^{-7/4} \log\frac{1}{\epsilon}$, and under the further assumption of
that $f$ has Lipschitz third derivatives, we show the same 
method attains rates of $\epsilon^{-5/3} \log\frac{1}{\epsilon}$.  This
raises the main question we address in this paper: how much further can this
$\epsilon$ dependence can be improved, and what Lipschitz continuity
assumptions are necessary?

\ifdefined\usenewtable
\newcolumntype{L}[1]{>{\raggedright\arraybackslash}p{#1}}
\begin{table}[t]
  \begin{center}
    \vspace{6pt}
    \begin{tabular}{llll}
      \toprule
      $f$ has Lipschitz & Upper bound & Lower bound & Gap  \\ 
      \midrule
      $p$th order derivative
      & $\DeltaF \epsilon^{-(p+1)/p}$ \cite{BirginGaMaSaTo17}
      & $\DeltaF \epsilon^{-({p+1})/{p}}$ (Part I~\cite{NclbPartI})
      & $O(1)$ \\ 
      gradient and Hessian
      &  $\DeltaF \epsilon^{-7/4}\log\frac{1}{\epsilon}$ \cite{CarmonDuHiSi17}
      & $\DeltaF \epsilon^{-12/7}$ (Thm.~\ref{thm:final})
      & $\epsilon^{-1/28}$ \\
      $q$th-order deriv.\ for $q\le p$, $p\ge3$
      &  $\DeltaF \epsilon^{-5/3}\log\frac{1}{\epsilon}$
      \cite{CarmonDuHiSi17}
      & $\DeltaF \epsilon^{-8/5}$ (Thm.~\ref{thm:final})
      & $\epsilon^{-1/15}$ \\
      gradient + $f$ convex
      & $\sqrt{\DeltaF} \epsilon^{-1} \log\frac{1}{\epsilon}$
      (Prop.~\ref{prop:complexity-stationary-convex})
      & $\sqrt{\DeltaF} \epsilon^{-1}$ (Thm.~\ref{thm:convex-lb})
      & $\log\frac{1}{\epsilon}$ \\ 
      \bottomrule
    \end{tabular} 
  \end{center}
  \caption{\label{table:summary} The number of iterations required to find
    $\epsilon$-stationary points of high dimensional functions in terms of
    $\epsilon$ and the function value gap $\DeltaF \defeq f(x\ind{0}) -
    \inf_x f(x)$.  In the first row, the upper bounds are achieved by
    deterministic $p$th-order methods and the lower bounds apply to all
    randomized methods of arbitrary order. In the other rows, deterministic
    algorithms achieve the upper bounds and the lower bounds apply to
    deterministic first-order methods. The gap column ignores logarithmic
    factors.}
\end{table}
\fi
\ifdefined\usenewertable
\begin{table}[t]
  \begin{center}
    \begin{tabular}{lllll}
      \toprule
      Oracle & $f$ has Lipschitz & Upper bound & Lower bound & Gap  \\ 
      \midrule
      Gen. & $p$th-order derivative &  $O\left(\epsilon^{-(p+1)/p}\right)$ 
      \cite{BirginGaMaSaTo17} & $\Omega\left(\epsilon^{-({p+1})/{p}}\right)$ 
      \hyperlink{LinkPartI}{Part I}  & $O(1)$ \\ 
      F.O. & gradient and Hessian & 
      $\Otil{\epsilon^{-7/4}}$ \cite{CarmonDuHiSi17}
      & $\Omega\left(\epsilon^{-12/7}\right)$ Thm.~\ref{thm:final}
      & $\Otil{\epsilon^{-1/28}}$ \\
      F.O. & $q$th derivative $\forall q\le p$, $p\ge3$ %
      &  $\Otil{\epsilon^{-5/3}}$ \cite{CarmonDuHiSi17}
      & $\Omega\left(\epsilon^{-8/5}\right)$ Thm.~\ref{thm:final}
      & $\Otil{\epsilon^{-1/15}}$ \\
      F.O. & gradient + $f$ convex
      & $\Otil{\epsilon^{-1}}$ 
      Prop.~\ref{prop:complexity-stationary-convex}
      & $\Omega\left(\epsilon^{-1}\right)$ Thm.~\ref{thm:convex-lb}
      & $\Otil{1}$ 
      \\ 
      \bottomrule
    \end{tabular} 
    \vspace{6pt}

  \caption{\label{table:summary} The number of iterations required to find
    $\epsilon$-stationary points of high dimensional functions $f$, where 
    $f(x\ind{0})-\inf_x f(x) \le O(1)$. The first
    column indexes the type of oracle access: general (all derivatives) or
    first-order (function value and gradient). In the first row, deterministic 
    $p$th-order methods achieve the upper bounds, and the lower
    bounds apply to all randomized methods of arbitrary order. In the other
    rows, the lower bounds apply to all deterministic first-order methods, and 
    such methods achieve the upper bounds.}
  \end{center}
\end{table}
\else
\newcolumntype{L}[1]{>{\raggedright\arraybackslash}p{#1}}
\begin{table}[t]
  \caption{\label{table:summary}
    The number of iterations required to find $\epsilon$-stationary
    points of high dimensional functions}
  \begin{center}
    \vspace{6pt}
    \begin{tabular}{lllll}
      \toprule
      & $f$ has Lipschitz & Upper bound & Lower bound & Gap  \\ 
      \midrule
      A & $p$th order derivative &  $O\left(\epsilon^{-(p+1)/p}\right)$ 
      \cite{BirginGaMaSaTo17} & $\Omega\left(\epsilon^{-({p+1})/{p}}\right)$ 
      Part I \cite{NclbPartI}  & $O(1)$ \\ 
      B & gradient and Hessian &  $\Otil{\epsilon^{-7/4}}$ \cite{CarmonDuHiSi17} & $\Omega\left(\epsilon^{-12/7}\right)$ Thm.~\ref{thm:final} & $\Otil{\epsilon^{-1/28}}$ \\
      C & $q$th derivative $\forall q\le p$, $p\ge3$ %
      &  $\Otil{\epsilon^{-5/3}}$ \cite{CarmonDuHiSi17} & $\Omega\left(\epsilon^{-8/5}\right)$ Thm.~\ref{thm:final} & $\Otil{\epsilon^{-1/15}}$ \\
      D &gradient + $f$ convex & $\Otil{\epsilon^{-1}}$ 
      Prop.~\ref{prop:complexity-stationary-convex}~~ & 
      $\Omega\left(\epsilon^{-1}\right)$ Thm.~\ref{thm:convex-lb} & $\Otil{1}$ 
      \\ 
      \bottomrule
    \end{tabular} 
    \vspace{6pt}

  \parbox{0.95\textwidth}{
    \footnotesize{\textbf{Notes}~~ The bounds apply for functions $f$ with
      $f(x\ind{0}) - \inf_x f(x) \le O(1)$.  In row A the upper bounds are
      achieved by deterministic $p$th-order methods and the lower bounds 
      apply to all randomized methods of arbitrary order. In rows B--D the 
      upper (lower) bounds are achieved by (apply to all) deterministic 
      first-order methods.}}
  \vspace{-12pt}
  \end{center}
\end{table}
\fi

\subsection{Our contributions}

In Table~\ref{table:summary} we summarize our results, along with
corresponding known upper bounds. We establish lower bounds on the
\emph{worst-case oracle complexity} of finding $\epsilon$-stationary points, 
where algorithms may access $f$ only through queries to an information 
oracle, that returns the value and some
number of (or potentially all) derivatives of $f$ at the queried point.
A lower bound $T_\epsilon$ means that for every algorithm $\alg$, there
exists a function $f$ in the allowed function class (\eg functions
with $f(x\ind{0})-\inf_x f(x) \le \DeltaF$ and
$\SmGrad$-Lipschitz gradient) for which $\alg$ requires at least
$T_\epsilon$ oracle queries before
returning an $\epsilon$-stationary point of $f$.

In Part I~\cite{NclbPartI} of this series we prove that no algorithm, even
one given all derivatives of $f$ at each iteration, can improve on the
$\epsilon^{-2}$ rate of gradient descent for the class of functions with
bounded initial value and Lipschitz continuous gradient. Therefore, in
distinction with the convex case, acceleration of gradient descent for
non-convex optimization~\cite{CarmonDuHiSi17} fundamentally depends on
higher-order smoothness assumptions. We further show that, for the class of
functions with $p$th order Lipschitz derivatives, no method can improve the
rate $\epsilon^{-(p+1)/p}$ achieved by a $p$th-order
method~\cite{BirginGaMaSaTo17}. However, this does not get at the crux of
the issue we consider here---what is the best possible rate for
\emph{first-order} methods, given that higher-order derivatives are
Lipschitz?

In this paper we show that the $\epsilon$-dependencies we establish in our
work~\cite{CarmonDuHiSi17} are almost tight. More precisely, consider the
function class with $\Smp$-Lipschitz derivatives for all $q \in \{1, \dots,
p \}$, where $p \in \N$; for this class there does not exist a deterministic
first-order algorithm with iteration complexity 
better than $\epsilon^{-8/5}$. If $p = 2$
this complexity lower bound strengthens to $\epsilon^{-12/7}$. 
In the following diagram, we compare the exponents of $1/\epsilon$ in our 
lower bounds and known upper bounds (smaller is better).

\newcommand{\ticksize}{5 pt}
\newcommand{\highlightsize}{3 pt}
\newcommand{\bracketsize}{0.004}
\newcommand{\diagramfrac}[2]{$\displaystyle\frac{#1}{#2}$}
\newcommand{\diagramnum}[1]{$\displaystyle #1$}
\begin{center}
  \begin{tikzpicture}[x=1.6\textwidth,y=1.6\textwidth]

  \draw (3/2,0) -- (2,0);

  \draw[fill=violet!50, fill opacity=1, draw opacity=0] (8/5,-\highlightsize) 
  rectangle (5/3, \highlightsize);
  \draw[fill=teal!50, fill opacity=1, draw opacity=0] (12/7,-\highlightsize) 
  rectangle (7/4, \highlightsize);
  \filldraw (3/2,0) circle (2pt);
  \filldraw (2,0) circle (2pt);

  \foreach \x in {3/2, 8/5, 5/3, 12/7, 7/4, 2}
  \draw (\x ,\ticksize) -- (\x ,-\ticksize);

  \draw (3/2,0) node[below=\ticksize] {\diagramfrac{3}{2}};
  \draw (3/2,0) node[above=\ticksize] 
  {\footnotesize\parbox{\widthof{Newton's method}}{\centering
             cubic-regularized Newton's method \\ $p=2$}};
  \draw (8/5,0) node[below=\ticksize] {\diagramfrac{8}{5}};
  \draw (8/5,\ticksize)--(8/5+\bracketsize,\ticksize);
  \draw (8/5,-\ticksize)--(8/5+\bracketsize,-\ticksize);
  \draw (8/10+5/6,0) node[above=\ticksize, color=violet] 
  {\footnotesize\parbox{\widthof{Newton's method}}{\centering
      first-order methods \\ $p\ge 3$}};
  \draw (5/3,0) node[below=\ticksize] {\diagramfrac{5}{3}};
  \draw (5/3,\ticksize)--(5/3-\bracketsize,\ticksize);
  \draw (5/3,-\ticksize)--(5/3-\bracketsize,-\ticksize);
  \draw (12/7,0) node[below=\ticksize] {\diagramfrac{12}{7}};
  \draw (12/7,\ticksize)--(12/7+\bracketsize,\ticksize);
  \draw (12/7,-\ticksize)--(12/7+\bracketsize,-\ticksize);
  \draw (12/14+7/8,0) node[above=\ticksize,color=teal] 
  {\footnotesize\parbox{\widthof{Newton's method}}{\centering
      first-order methods \\ $p=2$}};
  \draw (7/4,0) node[below=\ticksize] {\diagramfrac{7}{4}};
  \draw (7/4,\ticksize)--(7/4-\bracketsize,\ticksize);
  \draw (7/4,-\ticksize)--(7/4-\bracketsize,-\ticksize);
  \draw (2,0) node[below=\ticksize] {\diagramnum{2}};
  \draw (2,0) node[above=\ticksize] 
  {\footnotesize\parbox{\widthof{gradient}}{\centering
      gradient descent \\ $p=1$}};
  \end{tikzpicture}
\end{center}

\noindent
Thus, we establish two separations. First, \emph{no} deterministic first-order 
method can 
achieve the rate of
convergence $\epsilon^{-3/2}$ of Newton's method. Second, the rate 
$\epsilon^{-5/3} \log
\frac{1}{\epsilon}$ we achieve~\cite{CarmonDuHiSi17} \emph{requires} the
assumption of Lipschitz third derivatives, as first-order methods assuming 
only Lipschitz Hessian must compute at least $\epsilon^{-12/7}$ function
values and gradients to find an $\epsilon$-stationary point.
We also show that the optimal rate for finding $\epsilon$-stationary points
of \emph{convex} functions with bounded initial value (\ie $f(x\ind{0}) -
\inf_x f(x) \le \DeltaF$) and $\SmGrad$-Lipschitz gradient is
$\wt{\Theta}(\sqrt{\SmGrad\DeltaF}\epsilon^{-1})$.\footnote{
  Given a bound $\norms{x\ind{0}-x\opt} \le D$
  where $x\opt \in \argmin f$, as is standard for
  convex optimization, the optimal rate is
  $\wt{\Theta}(\sqrt{\SmGrad D}\epsilon^{-1/2})$~\cite{Nesterov12b}. The two
  rates are not directly comparable.
}
Finding stationary points is thus fundamentally easier for convex
functions.

The starting point of our development is
Nesterov's~\cite[\S~2.1.2]{Nesterov04} ``worst function in the world,''
\begin{equation}
  \fNesterov(x) \defeq \frac{1}{2}(x_1 - 1)^2 
  + \half \sum_{i=1}^{d-1} (x_i - x_{i+1})^2,
  \label{eq:nesterov-chain}
\end{equation}
which is instrumental in proving lower bounds for convex
optimization~\cite{Nesterov04,WoodworthSr16,ArjevaniShSh17} due to its
``chain-like'' structure.
We
establish our $\epsilon^{-1}$ lower bound for finding stationary
points of convex functions by making a minor modification to the
construction~\eqref{eq:nesterov-chain}. To prove our larger lower bounds for
non-convex functions, we augment $\fNesterov$ with a non-convex separable
function $\sum_{i=1}^d \Upsilon(x_i)$, with $\Upsilon:\R\to\R$ carefully
chosen to render Nesterov's ``worst function'' even worse. 

\paragraph{Paper organization}
Throughout, we use \textsc{Pi}.$k$ to reference an item $k$ of Part I of
this sequence~\cite{NclbPartI}, as we build off of many ideas
there. In Section~\ref{sec:recap}
we briefly summarize our framework
(Sections~\exref{sec:prelims}
and~\exref{sec:anatomy}). Section~\ref{sec:convex} begins the new analysis
and contains lower bounds for finding stationary points of \emph{convex}
functions. In Section~\ref{sec:firstorder_construction} we construct our hard
non-convex instance, while in Section~\ref{sec:firstorder-lb} we
use this function
to establish our main result: a lower bound on the complexity of finding
stationary points using deterministic first-order
methods. %
In Sections~\ref{sec:discussion} we discuss some difficulties in sharpening or extending our lower bounds. Section~\ref{sec:concluding-remarks} concludes by situating our work in the current literature and reflecting on its implications for future research.

\paragraph{Notation}
Before continuing, we provide the conventions we adopt throughout the paper;
our notation mirrors Part I~\cite{NclbPartI}, so we describe it only
briefly. For a sequence of vectors, subscripts denote coordinate index,
while parenthesized superscripts denote element index, \ie $x\ind{i}_j$ is the
$j$th coordinate of the $i$th entry in the sequence $\{x\ind{t}\}_{t\in\N}$.  For 
any $p\ge1$ and $p$ times continuously differentiable $f:\R^d
\to \R$, we let $\deriv{p} f(x)$ denote the symmetric tensor of $p$th order 
partial
derivatives of $f$ at point $x$. We let $\inner{\cdot}{\cdot}$ be the Euclidean
inner product on tensors, defined for order $k$ tensors $T$ and $M$ by
$\inner{T}{M} = \sum_{i_1, \ldots, i_k} T_{i_1, \ldots, i_k}M_{i_1, \ldots,
  i_k}$. We use $\otimes$ to denote the Kronecker product and
$\tensordim{k}{d}$ denote $d \times \cdots \times d$, $k$ times,
so that $T \in \R^{\tensordim{k}{d}}$ denotes an order $k$ tensor.

For a vector $v\in\R^d$ we let $\norm{v} \defeq \sqrt{\inner{v}{v}}$ denote the Euclidean ($\ell_2$) norm of $v$.  
For a tensor $T \in
\R^{\tensordim{k}{d}}$, the $\ell_2$-operator norm of $T$ is
$\opnorm{T} \defeq \sup_{\norm{v\ind{i}} \le 1}
\<v\ind{1} \otimes \cdots \otimes v\ind{k}, T\>$, where
we recall~\cite{ZhangLiQi12} that if $T$ is symmetric then
$\opnorm{T} = \sup_{\norm{v} = 1} |\<v^{\otimes k}, T\>|$
where $v^{\otimes k}$ denotes the $k$-th Kronecker power of $v$.
For vectors the
$\ell_2$ and $\ell_2$-operator norms are identical. %

For any $n\in\N$, we let $[n] \defeq \{1, \ldots, n\}$ denote the set of
positive integers less than or equal to $n$. We let $\mc{C}^\infty$ denote
the set of infinitely differentiable functions. We denote the $i$th standard
basis vector by $e\ind{i}$, and let $I_d \in \R^{d\times d}$ denote the $d
\times d$ identity matrix; we drop the subscript $d$ when it is clear from
context.  For any set $\mc{S}$ and functions $g,h : \mc{S} \to
\openright{0}{\infty}$ we write $g \lesssim h$ or $g = O(h)$ if there exists
a numerical constant $c < \infty$
such that $g(s) \le c\cdot h(s)$ for every $s\in\mc{S}$. We
write $g = \Otil{h}$ if $g \lesssim h \log(h + 2)$.

\section{A framework for lower bounds}
\label{sec:recap}

For ease of reference, this section provides a condensed version of
Sections~\exref{sec:prelims} and~\exref{sec:anatomy} of the first part of this 
series~\cite{NclbPartI} that lays out the notation, concepts
and strategy we use to prove lower bounds. Here, we are deliberately
brief; see~\cite{NclbPartI} for motivation, intuition and background 
for our definitions, 
as well as exposition of randomized and higher-order methods.

\subsection{Function classes}\label{sec:recap-funcs}

Typically, one designs optimization algorithms for certain classes of
appropriately regular functions~\cite{Nesterov04, BoydVa04, NemirovskiYu83}.
We thus focus on two notions of regularity that have been important for both
convex and non-convex optimization:
Lipschitzian properties of derivatives and bounds on function value.  A
function $f:\R^d \to \R$ has $\Smp$-Lipschitz $p$th order derivatives if it
is $p$ times continuously differentiable, and for every $x \in \R^d$ and
$v \in \R^d$, $\norm{v} = 1$, the directional projection
$t \mapsto f_{x,v}(t) \defeq f(x + t \cdot v)$ of $f$
satisfies
\begin{equation*}
  \left| f_{x,v}^{(p)}(t) - f_{x,v}^{(p)}(t') \right|
  \le \Smp \left|t - t'\right|
  ~~ \mbox{for~} t, t' \in \R,
\end{equation*}
where $f_{x,v}^{(p)}(\cdot)$ is the $p$th
derivative of $t \mapsto f_{x,v}(t)$. We occasionally refer to a
function with Lipschitz $p$th order derivatives as $p$th-order smooth.

\begin{definition}
  \label{def:intersect-f-class}
  Let $p \ge 1$, $\DeltaF > 0$ and $\Sm{p} > 0$. Then the set
  \begin{equation*}
    \Fclass{p}
  \end{equation*}
  denotes the union, over $d \in \N$, of the collection of 
  $\mc{C}^\infty$ functions $f : \R^d \to \R$ with $\Smp$-Lipschitz $p$th
  derivative and $f(0)-\inf_x f(x) \le \DeltaF$. For positive $\DeltaF$ and
  $\Sm{1}, \ldots, \Sm{p}$ we define
  \begin{equation*}
    \FclassMany{p} \defeq
    \bigcap_{q \le p}\Fclass{q}.
  \end{equation*}

\end{definition}
\noindent
The function classes $\Fclass{p}$ include functions on $\R^d$ for all $d \in
\N$, following the established practice of studying ``dimension free''
problems~\cite{NemirovskiYu83,Nesterov04,NclbPartI}.
 
We also require the following important invariance
notion~\cite[Ch.~7.2]{NemirovskiYu83}.
\begin{definition}[Orthogonal invariance]
  A class of functions $\mc{F}$ is \emph{orthogonally invariant} if for
  every $f \in \mc{F}$, $f : \R^d \to \R$, and every matrix $U \in \R^{d'
    \times d}$ such that $U^\top U = I_d$, the function $f_U: \R^{d'} \to
  \R$ defined by $f_U(x) = f(U^\top x)$ belongs to $\mathcal{F}$.
\end{definition}
\noindent
Every function class we consider is orthogonally invariant.

\subsection{Algorithm classes}\label{sec:recap-algs}

For any dimension $d\in \N$, an \emph{algorithm} $\alg$ (also referred to as
a \emph{method} or \emph{procedure}) maps functions $f:\R^d\to\R$ to a
sequence of \emph{iterates} in $\R^d$; that is, $\alg$ is defined separately
for every finite $d$. We let
\begin{equation*}
  \alg[f] = \{x\ind{t}\}_{t=1}^\infty
\end{equation*}
denote the sequence $x\ind{t} \in \R^d$ of iterates that $\alg$ generates
when operating on $f$.
Throughout this paper, we focus on \emph{first-order
deterministic algorithms}. Such an algorithm $\alg$ is one that,
operating on $f : \R^d \to \R$, produces iterates of the form
\begin{equation*}
  x\ind{i} = \alg\ind{i}
  \left(f(x\ind{1}), \grad f(x\ind{1}),
  \ldots,
  f(x\ind{i-1}), \grad f(x\ind{i-1})\right)
  ~~\text{for }i\in\N,
\end{equation*}
where $\alg\ind{i} : \R^{d(i-1)+i} \to \R^d$ is measurable
(the dependence on dimension $d$ is implicit). We denote the
collection of first-order deterministic algorithms by $\AlgDet\ind{1}$.

Key to our development are \emph{zero-respecting algorithms}
(see Sec.~\exref{sec:prelims-algs} for more information).
For $v \in\R^d$ we let $\support{v} \defeq \{ i \in [d] \mid v_i \neq 0 \}$
denote the support (non-zero indices) of $v$.
Then we say that the sequence
$x\ind{1}, x\ind{2}, \ldots$ is \emph{first-order zero-respecting
  with respect to $f$} if
\begin{equation}
  \label{eq:firstorder-prelim-zr-def}
  \support{x\ind{t}} \subseteq \bigcup_{s < t}
  \support{\grad f(x\ind{s})}~~\mbox{for each } t \in \N.
\end{equation}
The definition~\eqref{eq:firstorder-prelim-zr-def} says that $x\ind{t}_i =
0$ whenever the partial derivatives of $f$ with respect to coordinate $x_i$ are 
zero for all preceding iterations.  Extending the
definition~\eqref{eq:firstorder-prelim-zr-def} in the obvious way, an algorithm
$\alg$ is \emph{first-order zero-respecting} if for any $f : \R^d\to \R$, the
iterate sequence $\alg[f]$ is zero-respecting with respect to $f$. The set
$\AlgZR\ind{1}$ comprises all such first-order algorithms.

\subsection{Complexity measures}\label{sec:recap-complexity}

For a sequence $\{x\ind{t}\}_{t\in\N}$ we define
the \emph{complexity} of the sequence $\{x\ind{t}\}_{t \in \N}$ on
$f$ by
\begin{equation*}
  \TimeEps{\{x\ind{t}\}_{t\in\N}}{f} \defeq
  \inf\left\{t\in\N \mid \normbig{\grad
    f(x\ind{t})} \le \epsilon \right\},
\end{equation*}
the index of the first element in $\{x\ind{t}\}_{t\in\N}$ that is an
$\epsilon$-stationary point of $f$.
The complexity of algorithm $\alg$ on $f$ is simply
the complexity of the sequence $\alg[f]$ on $f$, so
we define
\begin{equation*}
  \TimeEps{\alg}{f} \defeq
  \TimeEps{\alg[f]}{f}.
\end{equation*}
We define the \emph{complexity of algorithm
  class $\mathcal{A}$ on function class $\mathcal{F}$} as
\begin{equation}
  \label{eqn:firstorder-complexity-def}
  \CompEps{\mathcal{A}}{\mathcal{F}} \defeq 
  \inf_{\alg\in\mc{A}}\sup_{f\in\FclassBlank} \TimeEps{\alg}{f}.
\end{equation}
Table~\ref{table:summary} provides upper and lower bounds on the
quantity~\eqref{eqn:firstorder-complexity-def} for different choices of
$\mc{A}$ and $\FclassBlank$. For example, gradient descent guarantees
$\CompEps{\AlgDetFO\cap\AlgZRFO}{\Fclass{1}} \le 2\DeltaF \Sm{1}
\epsilon^{-2}$.

\subsection{How to show a lower bound}
\label{sec:recap-anatomy}

The last step in our preliminaries is to give an overview of our proof
strategy; this is an abbreviated version of
Section~\exref{sec:anatomy}. There, we abstract classical techniques for
lower bounds in convex optimization~\cite{NemirovskiYu83,Nesterov04},
presenting a generic method for proving lower bounds on deterministic
methods (of any order) applied to functions in any orthogonally invariant
class.

Our starting point is what we call a \emph{zero-chain},
which distills the ``chain-like'' structure of Nesterov's
construction~\eqref{eq:nesterov-chain}.
\begin{definition}\label{def:recap-zero-chain}
  A function $f : \R^d \rightarrow \R$ is a \emph{first-order
    zero-chain} if for every $x \in \R^d$,
  \begin{equation*}
    \support{x} \subseteq 
    \{ 1, \dots, i - 1 \}
    ~~ \mbox{implies} ~~
    \support{\grad f(x)} \subset  \{ 1, \dots, i \}.
  \end{equation*}
\end{definition}
\noindent
In Definition~\exref{def:zero-chain}~\cite{NclbPartI}, we extend   
zero-chains to higher orders; in our terminology Nesterov's 
function~\eqref{eq:nesterov-chain} is 
a first-order zero-chain, but not a second-order zero-chain.
A first-order zero-chain limits the rate that zero-respecting
algorithms acquire information from derivatives, forcing them to ``discover'' 
coordinates one by one, as
the following observation makes clear.

\begin{observation}
  \label{obs:recap-zero-chain}
  Let $f : \R^d \rightarrow \R$ be a first-order zero-chain and let
  $x\ind{1} = 0, x\ind{2}, \ldots$ be a first-order zero-respecting sequence
  with respect to $f$. Then $x\ind{t}_j=0$ for $j \ge t$ and all $t\le d$.
\end{observation}

The important insight, essentially due to \citet{NemirovskiYu83}
 is that by using a resisting oracle~\cite{NemirovskiYu83,
  Nesterov04} that can adversarially rotate the function $f$, any
lower bound for zero-respecting algorithms implies an identical bound for
all deterministic algorithms:
\begin{proposition}
  \label{prop:recap-det-zr}
  Let $\mc{F}$ be an orthogonally invariant function class. Then
  \begin{equation*}
    \CompEps{\AlgDet\ind{1}}{\mathcal{F}} 
    \ge
    \CompEps{\AlgZR\ind{1}}{\mathcal{F}}.
  \end{equation*}
\end{proposition}
\noindent
See Proposition~\exref{prop:prelims-det-zr} for a more general version of
this result.

\newcommand{\fbad}{f_\epsilon}

With this proposition, our strategy, inspired by
\citet{Nesterov04}, becomes clear. To prove a lower bound on first-order 
deterministic algorithms for a
function class $\mc{F}$, we find $\fbad : \R^{\T}
\rightarrow \R$ such that {(i)} $\fbad$ is a first-order zero-chain, 
{(ii)} $\fbad \in \mc{F}$, and {(iii)} $\norm{\grad \fbad(x)} 
>\epsilon$ for every $x$ such that $x_{\T}=0$.  Then for 
$\alg\in\AlgZR\ind{1}$ and
$\{x\ind{t}\}_{t\in\N}=\alg[f]$, Observation~\ref{obs:recap-zero-chain}
shows that $x\ind{t}_T=0$ for $t\le \T$, and the large gradient
property~(iii) guarantees the non-stationarity $\norm{\grad \fbad(x\ind{t})}
>\epsilon$ for all $t\le\T$. We immediately obtain the complexity lower
bound
\begin{equation*}
  \CompEps{\AlgZR\ind{1}}{\FclassBlank}
  = \inf_{\alg\in\AlgZR\ind{1}}\sup_{f \in \FclassBlank}\TimeEps{\alg}{f}
  \ge \inf_{\alg\in\AlgZR\ind{1}}\TimeEps{\alg}{\fbad} > T.
\end{equation*}
This strategy highlights the importance of ``dimension freedom,'' because we
take the dimension of $\fbad$ to be at least $T$, which must thus grow
inversely with $\epsilon$.

\section{Lower bounds for finding stationary points of convex functions}
\label{sec:convex}

While for \emph{convex} optimization guarantees of small gradients
are atypical topics of study, we nonetheless begin by considering the
complexity of finding stationary points of smooth convex functions.  This
serves two purposes.
First, it is a baseline for finding stationary points in the non-convex
setting; based on algorithmic upper bounds due to \citet{Nesterov12b}, we
see that convexity makes this task fundamentally easier. Second, our lower
bound construction for convex problems underpins our construction and
analysis for general smooth (non-convex) functions in the sequel, allowing
us to demonstrate our techniques in a simpler setting.
Of course, in convex optimization, it is typically more useful to find
points $x$ with small optimality gap, 
$f(x) \le \inf_{z} f(z) +\epsilon$. 
Convexity allows efficient algorithms for guaranteeing such
optimality, and typically one ignores questions of the magnitude of the
gradient in favor of small optimality or duality
gaps~\cite{BoydVa04}. Nonetheless, in some situations---such as certifying
(near) dual feasibility or small constraint residuals in primal-dual or
operator splitting algorithms~\cite[e.g.][]{BoydPaChPeEc11}---achieving
small gradients is important.

We proceed as follows. In Section~\ref{sec:convex-classes} we define the
class of convex functions under consideration and a quadratic subclass. In
Section~\ref{sec:convex-contstruction}, we construct a hard quadratic
instance, and verify its key properties. Finally,  in Section~\ref{sec:convex-lb}, 
we state, discuss and prove our lower bounds.

\subsection{Convex function classes}\label{sec:convex-classes}

The collections of functions we consider are the following.
\begin{definition}
  \label{def:convex-sets}
  Let $\SmGrad > 0$ and $\DeltaF > 0$. The set
  \begin{equation*}
    \ConvClass
  \end{equation*}
  denotes the union, over $d \in \N$, of the collections of $\mc{C}^\infty$
  convex functions $f : \R^d \to \R$ with $\SmGrad$-Lipschitz gradient and
  $f(0) - \inf_x f(x) \le \DeltaF$.
  Additionally,
  \begin{equation*}
    \QuadClass \subset \ConvClass
  \end{equation*}
  is the set of convex quadratic functions satisfying the above conditions.
\end{definition}

Our results, following \citet{NemirovskiYu83} and \citet{Nesterov04},
demonstrate that for deterministic first-order methods, the class 
$\QuadClass$ is ``hard
enough,'' in that it provides nearly sharp lower bounds for first-order
methods, which immediately apply to $\ConvClass$ and $\Fclass{1}$. We also
have $\QuadClass = \FclassManyCustom{p}{\DeltaF, \SmGrad, 0, \ldots,
	0}$ or any $p \ge 2$.%

 In addition to functions restricted by initial optimality gap, we consider the following initial distance-based definition.
\begin{definition}
  \label{def:distance-convex-sets}
  Let $D > 0$ and $\SmGrad > 0$. The set
  \begin{equation*}
    \ConvClassD
  \end{equation*}
  denotes the union, over $d \in \N$, of the collections of $\mc{C}^\infty$
  convex functions $f : \R^d \to \R$ with $\SmGrad$-Lipschitz gradient
  satisfying $\norm{x\opt} \le D$ for all $x\opt \in \argmin_x f(x)$.
  Additionally, 
  \begin{equation*}
    \QuadClassD \subset \ConvClassD
  \end{equation*}
  is the set of convex quadratic functions satisfying the above
  conditions.
\end{definition}

Standard convergence results in (smooth) convex
optimization~\cite[e.g.][]{Nesterov04} apply to functions with bounded
domain, \ie $f \in \ConvClassD$ rather than $\ConvClass$. This 
is for good
reason: for any pair $\DeltaF, \SmGrad$, any $\epsilon <\DeltaF$, any
first-order zero-respecting or deterministic algorithm $\alg$, and
\emph{any} $T\in\N$, there exists a function $f \in \QuadClass$ with
$\SmGrad$-Lipschitz gradient such for $\{x\ind{t}\}_{t \in \N} = \alg[f]$ we
have
\begin{equation*}
  \inf_{t \in \N} \left\{t \mid f(x\ind{t}) \le \inf_x f(x) + \epsilon
  \right\} > T.
\end{equation*}
(See Appendix~\ref{sec:convex-takes-forever},
Lemma~\ref{lemma:convex-takes-forever} for a proof of this claim.)  Since
this holds for any $T\in\N$ and $\epsilon<\DeltaF$, making even the
slightest function value improvement to functions in $\QuadClass$ may take
arbitrarily long. Thus, when we consider the function classes of
Definition~\ref{def:convex-sets}, we can only hope to give convergence
guarantees in terms of stationarity---as is common in the non-convex case.

\subsection{The worst function in the (convex) world}
\label{sec:convex-contstruction}

We now constructing the functions that are difficult for any
zero-respecting first-order method.  For parameters $\T\in\N$
and $\alpha \le 1$ we define the (unscaled) hard function
$\fhardConv:\R^{\T}\to\R$ by
\begin{equation}
  \label{eq:fhard-conv-def}
  \fhardConv(x) =  \frac{\alpha}{2}(x_1 - 1)^2 
  + \half \sum_{i=1}^{\T-1} (x_i - x_{i+1})^2.
\end{equation}
For $\alpha=1$, $\fhardConv[\T,1]$ this is Nesterov's
``worst function in the world''
\cite[\S~2.1.2]{Nesterov04}.
The
parameter $\alpha$ allows us to control $f(0)$ and thus provides a degree of
freedom in satisfying the constraint $f(0)-\inf_x f(x) \le \Delta$ for our
lower bounds. By inspection,
\begin{equation*}
  \fhardConv(x)= \half x^\top L x - b^\top x + \frac{\alpha}{2} 
\end{equation*}
where
\begin{equation}
  \label{eqn:laplacian}
  L = \left[\begin{array}{ccccc}
      1+\alpha & -1 \\
      -1 & 2 & -1 \\
      & \ddots & \ddots & \ddots\\
      &  & -1 & 2 & -1\\
      &  &  & -1 & 1
    \end{array}\right]
  \in \R^{\T\times\T}
\end{equation}
is the unnormalized graph Laplacian of the simple path
on $\T$ vertices (see~\cite{Chung98}) plus the term $\alpha$ in the
position $L_{11}$, and $b = \alpha
e\ind{1}$.

Let us now verify that $\fhardConv$ meets the three requirements of our
lower bound strategy.

\begin{lemma}\label{lem:convex-props}
  For all $\T\in\N$ and $\alpha \le 1$, $\fhardConv$ has
  the following properties.
  \begin{enumerate}[i.]
  \item \label{item:convex-zero-chain} \textbf{Zero-chain}~~
    $\fhardConv$ is a first-order zero-chain.
  \item \label{item:convex-fclass} \textbf{Membership in function class}~~
    \begin{enumerate}[(a)]
    \item $\fhardConv$ has $4$-Lipschitz continuous gradient.
    \item $\fhardConv(0) - \inf_{x\in\R^{\T}}\fhardConv(x) = \alpha/2$.
    \item The unique 
      minimizer of $\fhardConv(x)$ is
      $x\opt = \ones$, and
      $\norm{x\opt} = \sqrt{\T}$.
    \end{enumerate}
  \item \label{item:convex-hard} \textbf{Large gradient}~~
    For every $x\in\R^{\T}$ such that $x_{\T}=0$,
    $\norms{\grad \fhardConv(x)} > \left(\T-1+\frac{1}{\alpha}\right)^{-3/2}$.
  \end{enumerate}
\end{lemma}
\begin{proof}
  Part~\ref{item:convex-zero-chain} is immediate from
  Definition~\ref{def:recap-zero-chain}, since for every $i\in[d]$, $\grad_i
  \fhardConv(x)=0$ whenever
  $x_{i-1}=x_i=x_{i+1}=0$. Part~\ref{item:convex-fclass} is also immediate,
  as $\fhardConv(\ones) = \inf_x \fhardConv(x) = 0$, and $\opnorm{L}\le 4$ 
  (apply the triangle inequity to $\norm{Lv}$). To establish 
  part~\ref{item:convex-hard}, we
  calculate the minimum value of $\norms{\grad \fhardConv(x)}^2$ obtainable
  by any vector $x\in\R^\T$ with $x_\T=0$. Letting $M = L [I_{\T - 1} ~
    0_{\T-1}]^\top \in \R^{\T \times (\T - 1)}$ be the matrix $L$
  of~\eqref{eqn:laplacian} with its last column removed and recalling $b =
  \alpha e\ind{1}$, this becomes the least squares problem, whose solutions 
  is the squared norm of the projection of $b$ to the (one-dimensional) 
  nullspace 
  of $M^\top$:
  \begin{equation}\label{eq:convex-least-squares}
    \inf_{x\in\R^\T, x_\T = 0}
    \norm{\grad \fhardConv(x)}^2 = 
    \inf_{v\in\R^{\T-1}} \norm{M v - b}^2 = b^\top
    \left(I_\T- M(M^{\top}M)^{-1}M^\top\right)b
    =\left(z^\top b\right)^{2},
  \end{equation}
  where 
  $z \in \R^{\T}$ is the
  unique (up to sign) unit-norm solution to $M^\top z = 0$. A calculation
  shows that
  \begin{equation*}
    z_{j}= \frac{j-1+\frac{1}{\alpha}}{\sqrt{\sum_{i=1}^{\T}(i-1+\frac{1}{\alpha})^{2}}}.
  \end{equation*}
  Substituting $z$ and $b$ into
  Eq.~\eqref{eq:convex-least-squares}, we have that $x_\T = 0$ implies
  \begin{equation*}
    \norm{\grad \fhardConv(x)}^2 \ge \frac{1}{{\sum_{i=1}^{\T}(i-1+\frac{1}{\alpha})^{2}}} > \frac{1}{(T-1+\frac{1}{\alpha})^3},
  \end{equation*}
  giving the result.
\end{proof}

\subsection{Scaling argument and final bound}\label{sec:convex-lb}

With our hard instance in place, we provide our lower
bounds for finding stationary points of convex functions. We note that the 
lower bound for the class $\QuadClassD$ also follows from 
the standard lower bounds on finding $\epsilon$-suboptimal points, since for 
every $q\in\QuadClassD$ an $\epsilon$-stationary point is also $\epsilon 
D$-suboptimal.

\begin{theorem}\label{thm:convex-lb}
  Let $\epsilon, \DeltaF, D$, and $\SmGrad$ be positive. Then
  \begin{subequations}
    \begin{equation}
      \CompEps{\AlgDetFO}{\ConvClass} \ge
      \CompEps{\AlgZRFO}{\QuadClass} \ge
      \frac{\sqrt{\SmGrad\DeltaF}}{4}\epsilon^{-1},
      \label{eqn:fval-convex-lower}
    \end{equation}
    and
    \begin{equation}
      \CompEps{\AlgDetFO}{\ConvClassD} \ge 
      \CompEps{\AlgZRFO}{\QuadClassD} \ge 
      \frac{\sqrt{\SmGrad D}}{2}\epsilon^{-1/2}.
      \label{eqn:distance-convex-lower}
    \end{equation}
  \end{subequations}
\end{theorem}

Let us discuss Theorem~\ref{thm:convex-lb} briefly.  \citet{Nesterov12b}
shows that for any $f \in \ConvClassD$, accelerated gradient descent applied
to a regularized version of $f$ yields a point $x$ satisfying $\norm{\nabla
  f(x)} \le \epsilon$ after at most $O(\sqrt{\SmGrad
  D}\epsilon^{-1/2}\log\frac{\SmGrad D}{\epsilon})$ iterations.  For 
  $f\in\ConvClass$,
a similar technique to Nesterov's, which we provide for completeness in
Appendix~\ref{sec:app-convex-upper-bound}, yields an upper complexity bound
of $O(\sqrt{\SmGrad \DeltaF}\epsilon^{-1}\log\frac{\SmGrad
  \Delta}{\epsilon^2})$. Thus, to within logarithmic
factors both bounds of Theorem~\ref{thm:convex-lb} are sharp.
It is illustrative to compare Theorem~\ref{thm:convex-lb} to our results for
\emph{non-convex} but smooth functions, and we do so in detail in
Sec.~\ref{sec:conclusion-implications}. The comparison shows that finding
stationary points of smooth convex functions with first-order methods is
fundamentally easier than finding stationary points of non-convex functions,
even with higher-order smoothness and using higher-order methods.

While we prove our lower bounds for the algorithm classes $\AlgDetFO$ and
$\AlgZRFO$, similar lower bounds apply to the collection $\AlgRand$ of
all randomized algorithms based on arbitrarily high-order 
derivatives, when applied to worst-case functions from function class
$\ConvClass$. While this is not our focus here,
using the techniques of \citet{WoodworthSr16}
and Section~\exref{sec:fullder-random}, it is possible to construct a
distribution $\P$ on $\ConvClass$ such that for \emph{any} $\alg \in
\AlgRand$, with high probability over $f \sim \P$ we have $\TimeEps{\alg}{f}
\gtrsim \sqrt{\SmGrad \DeltaF} \epsilon^{-1}$.  That is, neither
randomization nor higher-order derivative information can improve
performance on $\ConvClass$.
Such an extension fails for $\QuadClass$, as Newton's method
finds the global minimizer of every $f \in \QuadClass$ in one step. We
know of no tight lower bounds on the complexity of randomized
first-order methods for $\QuadClass$.

\subsection{Proof of Theorem~\ref{thm:convex-lb}}

As we outline in Section~\ref{sec:recap-anatomy}, we establish our lower
bounds by constructing a zero-chain $f:\R^\T\to\R$ such that
$f\in\QuadClass$ (or $\QuadClassD$), and that $\norms{\grad f(x)} >
\epsilon$ for any $x$ such that $x_\T=0$. By
Observation~\ref{obs:recap-zero-chain} we immediately have that for every
$\alg\in\AlgZRFO$, the iterates $\{x\ind{t}\}_{t\in\N} = \alg[f]$ produced
by $\alg$ operating on $f$ satisfy $x\ind{t}_\T =0$ for every $t\le\T$ and
hence $\norms{\grad f(x\ind{t})} > \epsilon$. Consequently,
$\inf_{\alg\in\AlgZRFO}\TimeEps{\alg}{f} \ge 1+\T$, which implies lower bounds
on the required quantities by means of 
$\QuadClass \subset \ConvClass$ and
Proposition~\ref{prop:recap-det-zr}.

To define the difficult zero-chain $f$, we scale $\fhardConv$ using
two scalar parameters $\lambda, \sigma > 0$, which we determine
later, defining
\begin{equation*}
  f(x) \defeq \lambda \sigma^2 \fhardConv(x/\sigma).
\end{equation*} 
We use the parameter $\lambda > 0$ to control the first-order smoothness
of $f$, as $\nabla^2 f(x) = \lambda \nabla^2 \fhardConv(x/\sigma)$, while
the parameter $\sigma$ controls the lower bound on $\norm{\grad f(x)}$ for 
$x_T=0$. We first show how to
choose $\sigma$, depending on $\T$, $\epsilon$, $\alpha$, and $\lambda$. 
By
Lemma~\ref{lem:convex-props}.\ref{item:convex-hard}, for every $x$ with 
$x_T=0$ we have
\begin{equation*}
  \norm{\grad f(x)}
  = \lambda\sigma \norm{\grad \fhardConv(x)}
  > \frac{\lambda\sigma}{(T-1+\frac{1}{\alpha})^{3/2}}.
\end{equation*}
Setting
\begin{equation*}
  \sigma = \frac{1}{\lambda}\left(T-1+\frac{1}{\alpha}\right)^{3/2}\epsilon,
\end{equation*}
guarantees $\norms{\grad f(x)} > \epsilon$ for any $x$ such that
$x_\T=0$ and hence $\norms{\grad f(x\ind{t})} > \epsilon$ for all $t\le \T$.

All that remains is to choose $\lambda, \T$, and $\alpha$ to guarantee that
$f$ belongs to the appropriate quadratic class.  By
Lemma~\ref{lem:convex-props}.\ref{item:convex-fclass}, $f$ has $4\lambda$
Lipschitz gradient, so we take
\begin{equation*}
  \lambda = \SmGrad/4
\end{equation*}
and guarantee that $f$ has $\SmGrad$-Lipschitz gradient.  To guarantee that
$f\in \QuadClass$,
Lemma~\ref{lem:convex-props}.\ref{item:convex-fclass} yields
\begin{equation*}
  f(0) - \inf_x f(x) = \lambda \sigma^2 \alpha/2 =
  \frac{2\alpha}{\SmGrad}\left(T-1+\frac{1}{\alpha}\right)^{3}\epsilon^2,
\end{equation*}
where we have substituted our choice of $\sigma$ and $\lambda$ in the final
equality. Defining
\begin{equation*}
  \alpha = 1/T \le 1 ~~\mbox{we obtain}~~
  f(0) - \inf_x f(x) \le 16 T^2 \epsilon^2 / \SmGrad,
\end{equation*}
so to guarantee $f(0) - \inf_x f(x) \le \DeltaF$, it suffices to choose
\begin{equation*}
  \T = \floor{\frac{\sqrt{\SmGrad\DeltaF}}{4}\epsilon^{-1}}.
\end{equation*}
This gives the first part~\eqref{eqn:fval-convex-lower} of the theorem.  For
inequality~\eqref{eqn:distance-convex-lower}, we must have $f\in
\QuadClassD$. Let $x\opt = \sigma\mathbf{1}$ denote the minimizer of $f$, so
that
\begin{equation*}
  \norm{x^\star} = \sigma\sqrt{\T} =  \frac{4}{\SmGrad}\left(T-1+\frac{1}{\alpha}\right)^{3/2}\epsilon\sqrt{T},
\end{equation*}
where again we have substituted our choices of $\sigma$ and $\lambda$ in the
final equality. Consequently, to guarantee $\norm{x\opt} \le D$ it
suffices to take
\begin{equation*}
  \alpha = 1 ~~\mbox{and}~~
  \T = \floor{\frac{\sqrt{\SmGrad D}}{2}\epsilon^{-1/2}},
\end{equation*}
giving the bound~\eqref{eqn:distance-convex-lower}.

\section{Constructing the non-convex hard 
instance}\label{sec:firstorder_construction}

We now relax the assumption of convexity, and design a first-order zero-chain 
that provides bounds stronger than those of
Theorem~\exref{thm:fullder-final}, when we
restrict the algorithm class to first-order methods. The basis of our
construction is the convex zero-chain~\eqref{eq:fhard-conv-def}, which we
augment with non-convexity to strengthen the gradient lower bound in
Lemma~\ref{lem:convex-props}.\ref{item:convex-hard}, while ensuring that all
derivatives remain Lipschitz continuous. With this in mind, for each $T \in
\N$, we define the unscaled hard instance $\fhardFO: \R^{\T+1} \to \R$ as
\begin{equation}
  \label{eq:fhard-def}
  \fhardFO(x) = \frac{\sqrt{\mu}}{2}(x_1 - 1)^2 
  + \frac{1}{2 }\sum_{i=1}^{\T} (x_{i+1} - x_i)^2 
  + \mu \sum_{i=1}^{\T} \Ups(x_i).
\end{equation}
where the non-convex function $\Ups:\R\to\R$, parameterized by $r\ge1$, is
\begin{equation}
\label{eq:ups-def}
\Ups(x) = 120\int_1^x \frac{t^2(t-1)}{1+(t/r)^2}dt.
\end{equation}
We illustrate the construction $\fhardFO$ in Figure~\ref{fig:construction}; it is  
the sum of the convex hard instance~\eqref{eq:fhard-conv-def} (with 
$\alpha=\sqrt{\mu}$) and a separable non-convex function. In the following 
lemma, which we prove in Appendix~\ref{sec:props-proof}, we list the 
important properties of $\Ups$.

\begin{figure}
	\begin{minipage}{0.5\textwidth}
		\begin{center}
			\includegraphics[width=1\textwidth]{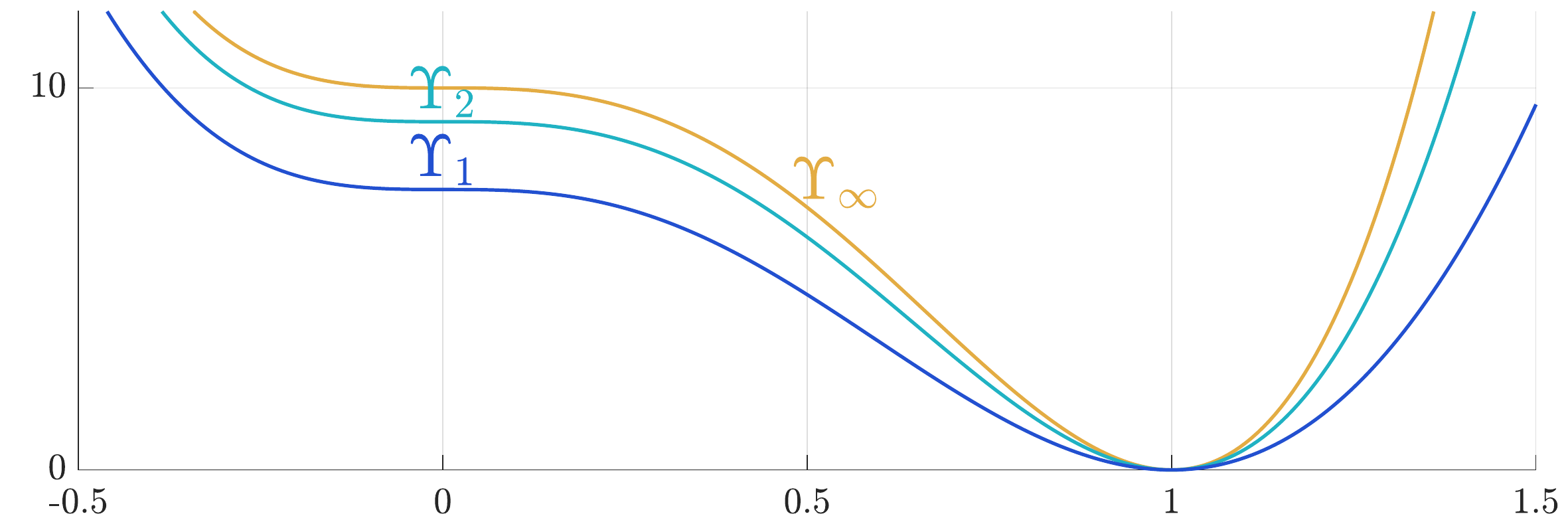}
			
			\includegraphics[width=1\textwidth]{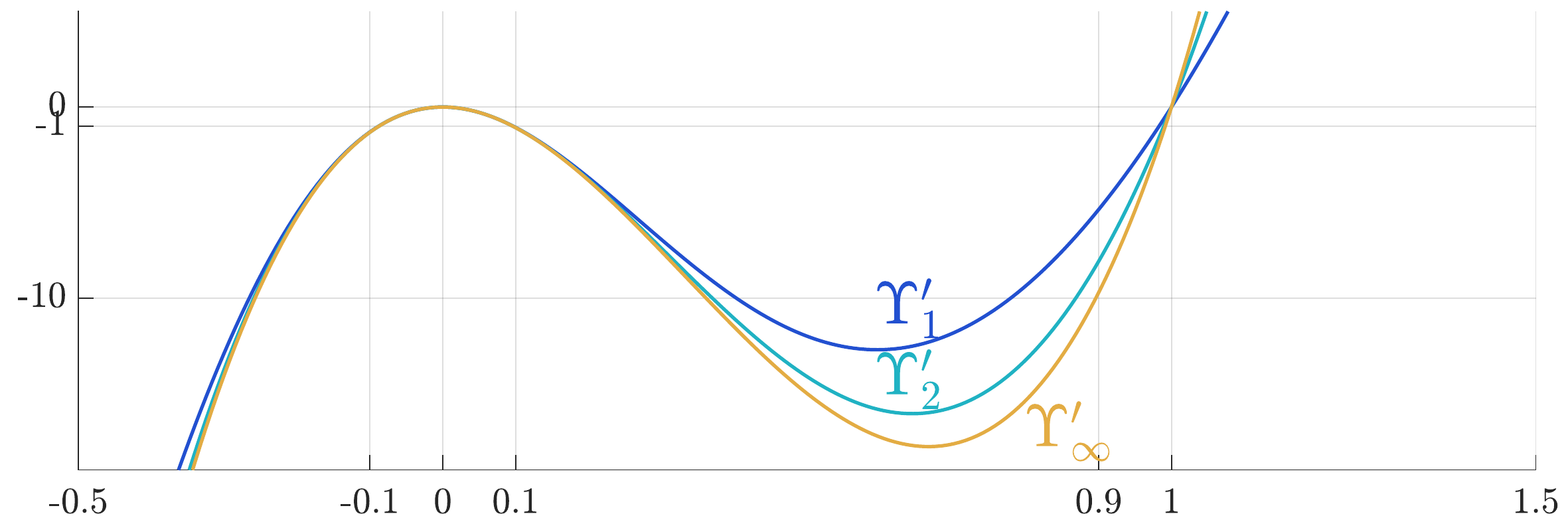}
			
		\end{center}
	\end{minipage} %
	\begin{minipage}{0.5\textwidth}
		\begin{center}
			\includegraphics[width=1\textwidth]{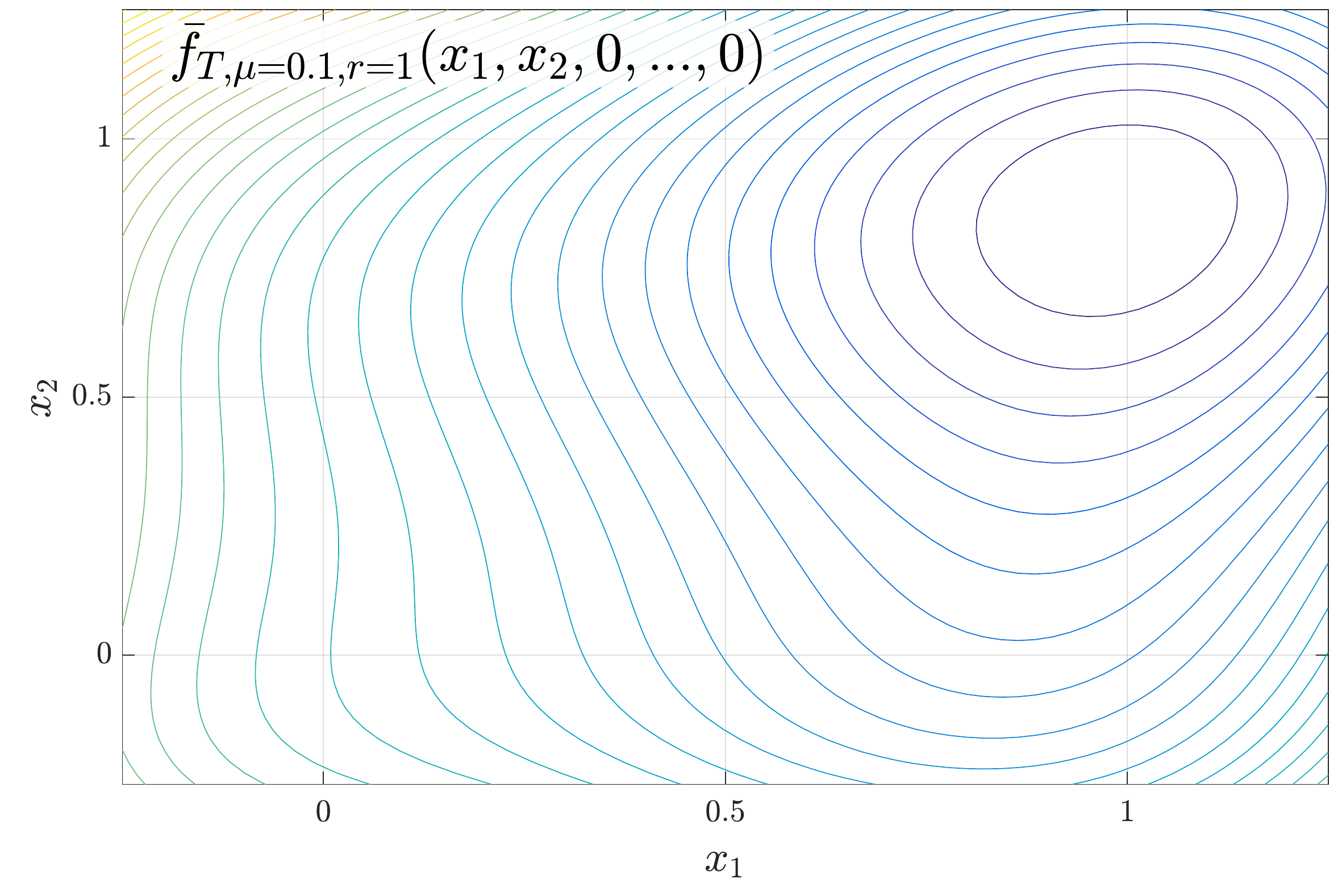}
		\end{center}
	\end{minipage} 
	\vspace{6pt}
	\caption{\label{fig:construction} Hard instance for first-order
		methods. Left: the non-convexity $\Ups$ (top) and its derivative
		(bottom), for different values of $r$. Right: Contour plot of a
		two-dimensional cross-section of the hard instance $\fhardFO$. }
\end{figure}

\begin{restatable}{lemma}{lemUpsProps}
  \label{lem:props}
   The function $\Ups$ satisfies the following.
  \begin{enumerate}[i.]
  \item \label{item:ups-grad-zero}
    We have $\Ups'(0) = \Ups'(1) = 0$
  \item \label{item:ups-quasi-convex}
    For all $x \le 1$, $\Ups'(x) \le 0$, and
    for all $x \ge 1$, $\Ups'(x) \ge 0$.
  \item \label{item:ups-min-one} For all $x \in \R$ we have $\Ups(x) \ge
    \Ups(1) = 0$, and for all $r$, $\Ups(0) \le 10$.
  \item \label{item:ups-large-grad}
    For every $r\ge 1$, $\Ups'(x) < -1$ for every $x\in (-\infty,-0.1]
    \cup [0.1,0.9]$
  \item \label{item:ups-c-infinity} For every $r \ge 1$ and every $p \ge 1$,
    the $p$-th order derivatives of $\Ups$ are $r^{3-p}\smC{p}$-Lipschitz
    continuous, where $\smC{p} \le \exp(\frac{3}{2}p \log p + c p)$ for a
    numerical constant $c < \infty$.
  \end{enumerate}
\end{restatable} 

Before formally stating the properties of $\fhardFO$, we provide a
high-level explanation of the choice of $\Ups$. First, a necessary and
sufficient condition for $\fhardFO$ to be a first-order zero-chain is that
$\Ups'(0)=0$. Second, examining the proof
Lemma~\ref{lem:convex-props}.\ref{item:convex-hard} we see that the gradient
of the quadratic chain is smallest for vectors $x$ with entries $x_1, x_2,
..., x_\T$ that slowly decrease from 1 to 0. We design $\Ups$ to ``punish''
such slowly varying vectors, by demanding that $\Ups'(x)$ be large for any
$x$ far from both 0 and 1 (Lemma~\ref{lem:props}.\ref{item:ups-large-grad});
this is the key to improving
Lemma~\ref{lem:convex-props}.\ref{item:convex-hard} and the most important
property of $\Ups$. Third, for every finite $r$ all the derivatives of
$\Ups$ are Lipschitz, and as $r$ increases $\fhardFO$ converges to a quartic
polynomial; in the limit $r=\infty$ we have $\Upsilon_\infty(x) = 30x^4 -
40x^3 +10$. This allows us to establish that Lipschitz continuity of
derivatives beyond the third does not alter the $\epsilon$ dependence of our
bounds. However, we cannot simply use $\Upsilon_\infty$, as its first three
derivatives are unbounded. Lastly, we place the minimum of $\Ups(x)$ at $x=1$,
so that the all-ones vector is the global minimizer of $\fhardFO$, and
$\fhardFO(\ones)=0$; this is simply convenient for our
analysis.

With our considerations explained, we verify the three components of our
general strategy: $f$ is a first-order zero-chain, belongs to the relevant
function classes, and has large gradient whenever $x_\T = 0$. We begin with
the zero-chain property, which follows trivially from
Lemma~\ref{lem:props}.\ref{item:ups-grad-zero}.

\begin{observation}\label{obs:span}
  For any $\T\in\N$, and positive $\mu$ and $r$, $\fhardFO$ is a first-order
  zero-chain.
\end{observation}
\noindent
Crucially, $\fhardFO$ is \emph{only} a first-order zero-chain (see 
Definition~\exref{def:zero-chain}); were it a second-order zero-chain, the 
resulting lower bounds would apply to
second-order algorithms as well, where Newton's method achieves the rate 
$\epsilon^{-3/2}$~\cite{NesterovPo06}, which is
strictly better than all of our lower bounds.
We next show that any point $x$ for which $x_T=x_{T+1}=0$ has large 
gradient. This is the core technical result of our analysis.

\begin{restatable}{lemma}{lemFirstorderGradbound}\label{lem:gradbound}
  Let $r\ge 1$ and $\mu \le 1$.  For any $x\in \R^{\T+1}$ such that $x_T =
  x_{T+1} = 0$,
  \begin{equation*}
    \norm{\grad \fhardFO(x)} > \mu^{3/4}/4.
  \end{equation*}
\end{restatable}
\noindent
We defer the full proof of this lemma to Appendix~\ref{sec:gradbound-proof}
and sketch its main idea here. We may view any vector meeting the conditions
of the lemma as a sequence going from $x_0 \defeq 1$ to $x_T = 0$. Every
such sequence must have a ``transition region'', which we define roughly as
the subsequence starting after the last $i$ such that $x_i > \frac{9}{10}$
and ending at the first (subsequent) $j$ such that $x_j <\frac{1}{10}$ (see
Figure~\ref{fig:transition-simple}). Letting $m\in\{1,\ldots,\T\}$ denote the
length of this subsequence and ignoring constant factors, we establish that
\begin{equation*}
\norm{\grad \fhardFO(x)} \ge \max\left\{ {\left(m+{1}/{\sqrt{\mu}}\right)^{-3/2}}\,,\,
\mu\sqrt{m}\right\}.
\end{equation*}
The $(m+1/\sqrt{\mu})^{-3/2}$ bound comes from the
quadratic chain in $\fhardFO$, which has large gradient for any
sequence $x$ with sharp transitions; this is essentially Lemma~\ref{lem:convex-props}.\ref{item:convex-hard} with $\T=m$ and $\alpha=\sqrt{\mu}$. The $\mu\sqrt{m}$ bound is due to the
non-convex $\Ups$ terms in $\fhardFO$, which by
Lemma~\ref{lem:props}.\ref{item:ups-large-grad} contribute a
term of magnitude $\mu$ to every entry of $\grad \fhardFO$ in the transition
region. These two bounds intersect at $m \approx 1/\sqrt{\mu}$, so the
gradient has norm at least $\mu^{3/4}$ for every value of $m$.

\begin{figure}
  \begin{center}
    \begin{tabular}{cc}
      \hspace{-.3cm}
      \includegraphics[width=0.48\textwidth]{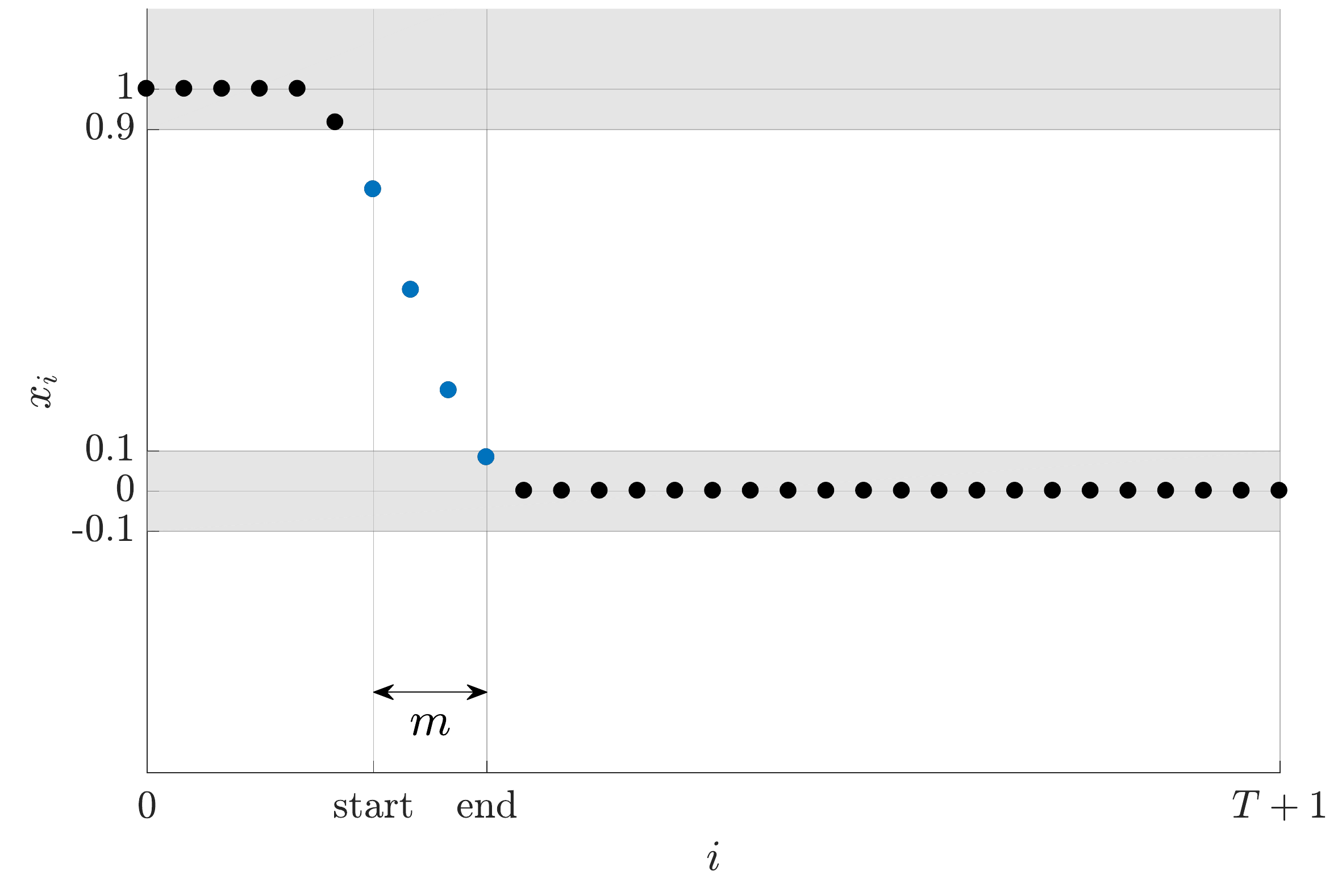} 
      & 
      \hspace{-.3cm}
      \includegraphics[width=0.48\textwidth]{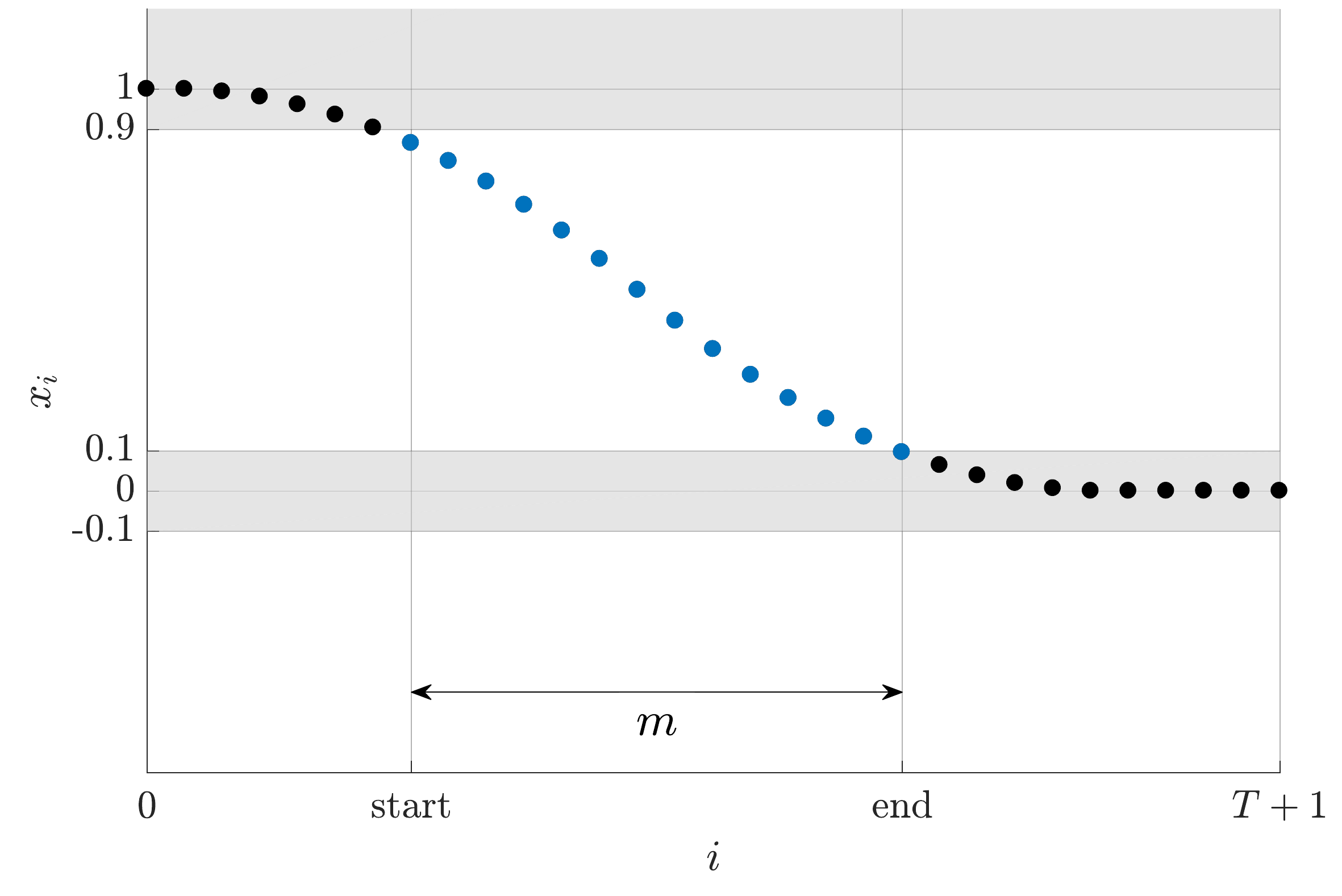} 
    \end{tabular}
    \caption{\label{fig:transition-simple}
      Illustration of the ``transition region'' concept used to prove
      Lemma~\ref{lem:gradbound}. Each plot shows the entries of a
      vector $x\in\R^{T+1}$ that satisfies $x_\T = x_{\T+1} = 0$, with
      entries of $x$ belonging to the transition region marked in
      blue. Short transitions (left) incur large gradients due to the convex
      quadratic term in $\fhardFO$, while long transitions (right) incur
      large gradients to the non-convex $\Ups$ terms and part 4 of
      Lemma~\ref{lem:props}. }
  \end{center}
\end{figure}

Finally, we list the boundedness properties of our construction.
\begin{restatable}{lemma}{lemFirstorderBounded}\label{lem:bounded}
  The function $\fhardFO$ satisfies the following.
  \begin{enumerate}[i.]
  \item \label{item:bounded-funcs}
    $\fhardFO(0) - \inf_x \fhardFO(x) \le \frac{\sqrt{\mu}}{2} + 10\mu T$
  \item \label{item:bounded-derivs}
    For $\mu\le1, r \ge 1$ and every $p\ge1$, the $p$-th order derivatives of
    $\fhardFO$ are $(\indic{p = 1} + r^{3-p}\mu) \smC{p}$-Lipschitz
    continuous, where $\smC{p} \le e^{\frac{3p}{2} \log p + cp}$ for
    a numerical constant $c < \infty$.
  \end{enumerate}
\end{restatable}
\begin{proof}
  The first part of the lemma
  follows from Lemma~\ref{lem:props}, which shows
  that $\inf_x \fhardFO(x) = \fhardFO(\ones) = 0$,
  while $\fhardFO(0) = \sqrt{\mu}/2 + \T\mu\Ups(0) \le \sqrt{\mu}/2 + 10 \mu T $.
  The second part of the lemma follows directly from
  Lemma~\ref{lem:props}.\ref{item:ups-c-infinity} and that the
  quadratic chain $f(x) = \frac{\sqrt{\mu}}{2}  (x_1 - 1)^2 +\half \sum_i (x_i - x_{i +
    1})^2$ has 4-Lipschitz gradient and 0-Lipschitz higher order derivatives. 
\end{proof}

\section{Lower bounds for first-order methods}\label{sec:firstorder-lb}

We now give our main result: lower bounds for the 
complexity of finding $\epsilon$-stationary using the class 
$\AlgDetFO\cup\AlgZRFO$ of first-order deterministic 
and/or zero-respecting algorithms, applied to functions in the class
\begin{equation*}
  \FclassMany{p} \defeq
  \bigcap_{q \le p}\Fclass{q}
\end{equation*}
containing all functions $f:\R^d\to\R$, $d \in \N$, such that $f(0)-\inf_x
f(x) \le \Delta$ and $\deriv{q} f$ is $\Sm{q}$-Lipschitz continuous for
$1 \le q \le p$.

\begin{restatable}{theorem}{thmFirstorderFinal}\label{thm:final}
  There exist numerical constants $c,C \in \R_+$ and $\smC{q} \le e^{\frac{3 q}{2} \log q + C q}$ for every $q\in\N$ such 
  that the following lower bound holds.
  Let $p\ge 2$, $p \in \N$,
  and let $\DeltaF, \SmGrad, \SmHess, \ldots, \Smp, \epsilon$ be
  positive. Assume additionally
  that $\epsilon \le (\SmGrad^q/\Sm{q})^{1/(q-1)}$ for each
  $q\in\{2,\ldots,p\}$. Then
  \begin{equation*}
    \CompEps{\AlgDetFO\cup\AlgZRFO}{\FclassMany{p}} \ge 
    c \cdot \DeltaF \cdot \min_{q\ \in \{2, \ldots, p\} }
    \left\{
    \left(\frac{\SmGrad}{\smC{1}}\right)^{\frac{3}{5}-\frac{2}{5(q-1)}}
    \left(\frac{\Sm{q}}{
      \smC{q}}\right)^{\frac{2}{5(q-1)}}
    \right\} \epsilon^{-8/5}.
  \end{equation*}
  Moreover, for $p=2$, 
  \begin{equation*}
    \CompEps{\AlgDetFO\cup\AlgZRFO}{\FclassTwo{1}{2}} \ge 
    c\cdot\DeltaF \left(\frac{\SmGrad}{\smC{1}}\right)^{\frac{3}{7}} 
    \left(\frac{\SmHess}{\smC{2}}\right)^{\frac{2}{7}} {\epsilon^{-{12}/{7}}}.
  \end{equation*}
\end{restatable}

We prove Theorem~\ref{thm:final} in Section~\ref{sec:proof-final} to come,
providing a brief overview of the argument here, and then providing some
discussion.  In the proof, we construct the hard instance $f: \R^{T+1} \to
\R$ as $f(x) = \lambda \sigma^2 \fhardFO(x/\sigma)$, where we must choose
the parameters $\lambda, \sigma > 0$ as well as $\mu, r$, and $\T$ to
guarantee that $f$ is (a) hard to optimize, \ie $\TimeEps{\alg}{f}>T$ for
every $\alg\in\AlgZRFO$, and (b) meets the smoothness and boundedness
requirements of the function class.

We begin by sketching the argument for $p=2$.  In this case, we may take
$r=1$, $\lambda \propto \SmGrad$ and $\mu
\propto \SmHess\sigma / \lambda$. This choice guarantees that $f$ has
$\SmGrad$-Lipschitz gradient and $\SmHess$-Lipschitz Hessian when $\mu \le
1$, which we later verify using the assumption $\epsilon\le
\SmGrad^2/\SmHess$. We then use Observation~\ref{obs:span}, 
Lemma~\ref{lem:gradbound} and Observation~\ref{obs:recap-zero-chain} 
to show that $\TimeEps{\alg}{f} \ge T+1$
for every $\alg \in \AlgZRFO$ whenever $\lambda \sigma\mu^{3/4}/4 \ge 
\epsilon$, and conclude that $\sigma$ may scale as
$\lambda^{1/7} \epsilon^{4/7}$ (since $\mu \propto \SmHess \sigma / 
\lambda$). By
Lemma~\ref{lem:bounded}.\ref{item:bounded-funcs} we
have
\begin{equation*}
  f(0)-\inf_x f(x) \le \lambda \sqrt{\mu}\sigma^2/2 + 10\lambda \mu \sigma^2 T,
\end{equation*}
so we can take $T\propto \DeltaF/(\lambda \mu \sigma^2)\propto
\DeltaF/(\SmHess \sigma^3)$ to guarantee $f(0)-\inf_x f(x) \le \DeltaF$,
where we assume without loss of generality that $\lambda \sqrt{\mu}\sigma^2
\le \DeltaF$ (otherwise Theorem~\ref{thm:convex-lb} dominates our
bound). Substituting the expressions for $\sigma, \mu$ and $\lambda$ into
the expression for $T$ gives the result for $p=2$.

For $p \ge 3$ we require a more careful argument, as we must simultaneously
handle all orders of smoothness. 
 To do so, we let  $\mu = \bar{\mu} \sigma^2 / \lambda$ and $r = 
 \bar{r}/\sigma$, and show how to take $\bar{r}$ and $\bar{\mu}$ 
 independently 
of $\epsilon$ (depending only on $\Sm{1}, \ldots, \Sm{p}$). This  
allows us to obtain identical $\epsilon$-dependence for all $p\ge 3$.

To better understand the theorem, we give a few additional remarks.

\paragraph{Near-achievability of the lower bounds}

In the paper~\cite{CarmonDuHiSi17}, we propose the method ``convex until 
proven guilty,'' which augments Nesterov's accelerated 
gradient method with implicit negative curvature descent. For the function
classes
$\FclassTwo{1}{2}$ and $\FclassThree{1}{2}{3}$, it
achieves rates of convergence
$\Otils{\DeltaF\SmGrad^{1/2}\SmHess^{1/4}\epsilon^{-7/4}}$ and
$\Otils{\DeltaF\SmGrad^{1/2}\Sm{3}^{1/6}\epsilon^{-5/3}}$,
respectively. These results nearly
match our lower bounds in Theorem~\ref{thm:final};
in the case of $p = 2$, the gap (in terms of $\epsilon$) is
of order $\epsilon^{-\frac{1}{28}}\log\frac{1}{\epsilon}$, while for
$p \ge 3$, the gap is of order
$\epsilon^{-\frac{1}{15}}\log\frac{1}{\epsilon}$.
See further discussion in
Sec.~\ref{sec:conclusion-implications}.

\paragraph{Choice of function class}
The focus on the more restricted function classes
$\FclassMany{p}$---rather than the classes $\Fclass{p}$ we study in
Part~I~\cite{NclbPartI}---makes our lower bounds stronger, and it
is necessary for non-trivial results, since for any $p\ge2$ and
$\Delta,\Smp>0$, the class 
$\Fclass{p}$ contains functions impossible for first-order methods. Indeed,
the class $\QuadClass$ of $\DeltaF$-bounded $\SmGrad$-smooth
convex quadratics is a subset of $\Fclass{p}$ for \emph{any} $\SmGrad <
\infty$ and $\Smp > 0$. Therefore, by Theorem~\ref{thm:convex-lb},
\begin{equation*}
	\CompEps{\AlgDetFO\cup\AlgZRFO}{\Fclass{p}} \ge 
	\sup_{\SmGrad < \infty} \CompEps{\AlgDetFO\cup\AlgZRFO}{\QuadClass}
	\ge \sup_{\SmGrad < \infty}
	\frac{\sqrt{\SmGrad\DeltaF}}{4}\epsilon^{-1} = \infty.
\end{equation*}
We thus limit our scope to functions with smooth lower order derivatives.

\paragraph{Conditions on the accuracy $\boldsymbol{\epsilon}$}
In Theorem~\ref{thm:final} we require that $\epsilon^{q - 1} \le \SmGrad^q /
\Sm{q}$ for all $q\in\{2, \ldots ,p\}$. For each $q$, we may rewrite this as
$\Sm{q}^{1/ q} \DeltaF \epsilon^{-(1+q)/q} \le \SmGrad \DeltaF \epsilon^{-2}$.
In other words, these conditions ensure that $q$th order
regularization-based methods have stronger convergence guarantees than
gradient descent~\cite{BirginGaMaSaTo17,NclbPartI}.

\paragraph{The case $\boldsymbol{p=1}$}
We state our bounds in Theorem~\ref{thm:final} for $p \ge 2$. It is possible
to use the construction~\eqref{eq:fhard-def} to prove a lower bound of
$O(\DeltaF\SmGrad\epsilon^{-2})$ on the time necessary for a deterministic
first-order algorithm to find an $\epsilon$-stationary point for the class
$\Fclass{1}$.  As Theorem~\exref{thm:fullder-final} shows this lower bound
holds for all randomized high-order algorithms, we do not pursue
this.

\paragraph{The case $\boldsymbol{p=3}$} 
We can slightly strengthen our lower bound in the case $p=3$, making it
independent of $\SmHess$ for sufficiently small $\epsilon$. To achieve this
we set $r=1$ in the definition of $\Upsilon_r$, take $\lambda\mu \propto
\Sm{3}\sigma^2$, and argue that that the resulting construction has
$O(\sigma)$-Lipschitz continuous Hessian, and $\sigma$ tends to zero as
$\epsilon \to 0$. For sufficiently small $\epsilon$, we can then replace the minimum over $q \in \{2, 3\}$ in
the first claim of Theorem~\ref{thm:final} with $\Sm{1}^{2/5}
\Sm{3}^{1/5}\epsilon^{-8/5}$.

\vspace{\baselineskip}

\noindent
The commentary on Theorem~\exref{thm:fullder-simple} in
Section~\exref{sec:fullder-deterministic-statement} is relevant also to
Theorem~\ref{thm:final}. In particular, it provides discussion of the polynomial 
scaling 
of  $\smC{q}^{1/q}$ in $q$.

\subsection{Lower bounds based on distance to optimality}

For convex optimization problems, typical convergence guarantees depend on
the distance of the initial point to the globally
optimal set $\argmin_x f(x)$; the dependence on this distance may be
polynomial for general convex optimization
problems~\cite{NemirovskiYu83,Nesterov04}, while for smooth and strongly
convex problems, the convergence guarantees depend only logarithmically on
it. In the non-convex case, we can provide lower bounds that
depend on the distance rather than the gap $\DeltaF \defeq f(x\ind{0}) -
\inf_x f(x)$.  To that end, we consider the class
\begin{equation*}
  \FclassDMany{p}
\end{equation*}
functions with $\Sm{q}$-Lipschitz $q$th derivatives (for each $q \in
[p]$) and all global minima $x^\star$ satisfying $\norm{x^\star}\le D$.  We
obtain the bound, analogously to our
results in Section~\exref{sec:fullder-distance}, by ``hiding'' a sharp global
minimum near the origin.

To state the theorem, we require an additional piece of notation. Let
$B_\epsilon(\DeltaF, \Sm{1}, \ldots, \Sm{p})$ be the lower bound
Theorem~\ref{thm:final} provides on $\CompEps{\AlgDet \cup
  \AlgZRFO}{\FclassMany{p}}$, so
\begin{equation*}
  \CompEps{\AlgDet \cup
    \AlgZRFO}{\FclassMany{p}} \ge B_\epsilon(\DeltaF, \Sm{1}, \ldots,
  \Sm{p}),
\end{equation*}
where we take $B_\epsilon = 1$ if $\epsilon > 0$ is larger than the settings
Theorem~\ref{thm:final} requires.  Then by a reduction from our lower bounds
on the complexity of $\FclassMany{p}$, we obtain the following result.

\begin{restatable}{theorem}{thmFirstorderFinalDistance}\label{thm:final-d}
  There exists a numerical constant $c < \infty$ such that the following
  lower bound holds. Let $p\ge 2$, $p \in \N$, and let $D, \SmGrad, \SmHess,
  \ldots, \Smp$, and $\epsilon$ be positive. Then
  \begin{equation*}
    \CompEps{\AlgDetFO \cup \AlgZRFO}{\FclassDMany{p}}
    \ge B_\epsilon\left(\min_{q\in[p]}
    \left\{\frac{\Sm{q}}{2\smCvar{q}}D^{q+1}\right\},
    \frac{\Sm{1}}{2}, \frac{\Sm{2}}{2}, \ldots, \frac{\Sm{p}}{2}
    \right),
  \end{equation*}
  where $\smCvar{q} \le \exp(c q \log q + c)$.
\end{restatable}
\noindent
We prove Theorem~\ref{thm:final-d} in Appendix~\ref{sec:D-lb-proof}.
 Theorem~\ref{thm:final-d} shows that
the lower bounds of Theorem~\ref{thm:final} apply almost identically
(to constant factors), except that we replace the function
gap $\DeltaF$ in the lower bound with the quantity
$\min_{q \in [p]} L_q D^{q + 1}$. As the dependence of the
lower bound on $\epsilon$ does not change,  distance-based
assumptions seem unlikely to help in the design of
efficient optimization algorithms for non-convex functions.

\subsection{Proof of Theorem~\ref{thm:final}}
\label{sec:proof-final}

We have five parameters with which to scale our hard function; the
function $\fhardFO$ requires definition of the dimension $\T \in \N$,
multiplier $\mu \le 1$ on the $\Ups$ terms, and scalar $r \ge 1$ that 
trades between higher order ($r = \infty$) smoothness and lower order ($r =
1$) smoothness of $\Ups$. We additionally scale the function with
$\lambda > 0$ and a perspective term $\sigma > 0$,
defining
\begin{equation}
  \label{eq:scaled-hard-fo}
  f(x) \defeq \lambda\sigma^{2}\fhardFO\left(x/\sigma\right).
\end{equation}
We must choose
these parameters to guarantee the membership
\begin{equation*}
  f \in \FclassMany{p}.
\end{equation*}
This containment requires both bounded function values and derivatives,
for which we can provide sufficient conditions.  Recall the definition
$\smC{p} \le e^{\frac{3}{2}p \log p + cp}$ from Lemma~\ref{lem:bounded}
of the smoothness constant of $\fhardFO$. Then by
Lemma~\ref{lem:bounded}.\ref{item:bounded-derivs},
to guarantee that $f$ has $\Sm q$-Lipschitz $q$th order derivatives for
every $q\in[p]$ it suffices to choose $\lambda, r, \sigma$, and $\mu$
such that
\begin{gather}
  (\indic{q=1} + r^{3-q} \sigma^{1-q}\mu) \smC{q} \lambda \le \Sm{q}
  \mbox{ for every }q\in[p]. 
  \label{eq:lb-lq} 
\end{gather}
For the bounded values constraint $f(0) - \inf_x f(x) \le \DeltaF$, by 
Lemma~\ref{lem:bounded}.\ref{item:bounded-funcs} it suffices to take
\begin{equation}
  T = \floor{\frac{\DeltaF-\lambda\sqrt{\mu}\sigma^2/2}{10\lambda\mu\sigma^2}}
  	\label{eq:lb-T}
\end{equation}
Thus, so long as we choose the constants
$\mu, \sigma, \lambda, r$ to satisfy inequality~\eqref{eq:lb-lq},
the preceding choice of $T$ guarantees $f \in \FclassMany{p}$.

With this membership guaranteed, we consider the choices for $\lambda,
\mu$, and $\sigma$ such that after $\T$ %
iterations of a zero
respecting first-order method, we have $\norms{\nabla f(x)} \ge \epsilon$.
Indeed, Observations~\ref{obs:span} and~\ref{obs:recap-zero-chain} imply 
that if $x\ind{1} = 0,
x\ind{2}, \ldots$ are the sequence of iterates produced by applying any
zero-respecting (first-order) method to $f$, then $x\ind t_{\T}=x\ind
t_{\T+1}=0$ for all $t\le\T$.  Lemma~\ref{lem:gradbound} implies that
$\norms{\grad\fhardFO\left(x\ind{t} / \sigma\right)}> \mu^{3/4} / 4$ for any
such iterate. Therefore, if we choose
$\lambda > 0, \mu \le 1$, and $\sigma > 0$ such that
\begin{equation}
  \label{eq:lb-eps}
  \ \lambda\mu^{3/4}\sigma \ge 4 \epsilon,
\end{equation}
then $\norms{\grad f\left(x\ind t\right)} =\lambda\sigma\norms{
  \grad\fhard(x\ind t/\sigma)} >\lambda\mu^{3/4}\sigma/4\ge\epsilon$ for all 
  $t\le\T$.
We thus obtain the guarantee 
\begin{equation}
  \TimeEps{\AlgZRFO}{\FclassMany{p}}
  \ge \inf_{A\in\AlgZRFO}\TimeEps{A}{f}
  \ge T+1 \ge \frac{\DeltaF-\lambda\sqrt{\mu}\sigma^2/2}{10\lambda\mu\sigma^2}.
  \label{eq:lb-genlb}
\end{equation}
The same bound for the class $\AlgDetFO$ then follows from Proposition~\ref{prop:recap-det-zr}. 
Our strategy is now the obvious one: we select $\lambda > 0, 0 < \mu \le
1, r \ge 1$, and $\sigma > 0$ to satisfy the function membership
constraints~\eqref{eq:lb-lq} and the large gradient
guarantee~\eqref{eq:lb-eps}.  Substituting our choices into the
boud~\eqref{eq:lb-genlb} will then yield the lower bound in the
theorem. We begin with the general case $p\ge2$ and later provide a
tighter construction for $p=2$.

\paragraph{General smoothness orders}
To simplify the derivation, we define, for any $q\in[p]$
\begin{equation}
  \SmN q \defeq \Sm{q} / \smCN{q}
  ~~\mbox{where}~~
  \smCN{q} \defeq \begin{cases}
    2\smC{1} & q=1 \\
    \max\left\{ \smC q,4^{q-1}\left(2 \smC 1\right)^{q}\right\} & q>1,
  \end{cases}
  \label{eq:lb-maxlp}
\end{equation}
where we note that $\smCN{q} \le e^{\frac{3}{2} q \log q + c q}$ for some
numerical constant $c < \infty$, as $\smC{q} \le e^{\frac{3}{2} q \log q +
  cq}$ for a (possibly different) numerical constant.  In order to further
simplify our calculations, we then define
\begin{equation}
  \lambda = \SmN{1}, %
  ~
  \bar{\mu} \defeq \frac{\lambda \mu}{\sigma^2}
  ~\mbox{and}~
  \bar{r} \defeq \sigma r.
  \label{eq:lb-mubar}
\end{equation}
Substituting
these definitions into the constraints~\eqref{eq:lb-lq}, we see
that our choice of $\smCN{1} = 2 \smC{1}$ implies that
the constraint~\eqref{eq:lb-lq} holds whenever
\begin{equation}
  \label{eq:hit-the-bar}
  \bar{r}^{3-q}\bar{\mu} \le \SmN{q}
  \mbox{ for all } q\in[p].
\end{equation}
We choose $\bar{r}$ and $\bar{\mu}$ to guarantee that $f$ is appropriately
smooth; in the sequel, we will choose $\sigma$ and $\lambda$ so that the
gradient bound condition~\eqref{eq:lb-eps} holds.  In this
sense, we may choose $\bar{r}$ and $\bar{\mu}$ without consideration of
$\epsilon$.  Taking $\bar{r} = ({\SmN{1}/\bar{\mu}})^{1/2}$ guarantees the
inequality~\eqref{eq:hit-the-bar} holds for $q=1$. Substituting this
choice into the identical inequality for $q \in \{2, \ldots, p\}$
shows that we must have $\bar{\mu}^{(q-1)/2} \le
\SmN{q}\SmN{1}^{(q-3)/2}$ for each such $q$. Thus, the choice
\begin{equation*}
  \bar{\mu} = \SmN{1}\min_{q\in\{2,\ldots,p\}} \left( \SmN{q}/\SmN{1}\right)^{2/(q-1)}
\end{equation*}
satisfies inequality~\eqref{eq:hit-the-bar}, and consequently, the
smoothness condition~\eqref{eq:lb-lq} as well. We may therefore
write $\bar{\mu}$ and $\bar{r}$ as
\begin{equation*}
  \bar{\mu} = \SmN{1} \left(\SmN{\qstar}/\SmN{1}\right)^{\frac{2}{\qstar-1}}
  \mbox{ and }
  \bar{r} = \left(\SmN{1}/\SmN{\qstar}\right)^{\frac{1}{\qstar-1}}
  \mbox{, where }
  \qstar \defeq \argmin_{q\in\{2,\ldots,p\}} \left( {\SmN{q}}/{\SmN{1}}\right)^{\frac{1}{q-1}}.	
\end{equation*}

It remains to choose $\lambda, \sigma$, depending on $\epsilon$, to
guarantee our gradient lower bound condition~~\eqref{eq:lb-eps}
holds, \ie
$4 \epsilon \le \lambda\mu^{3/4}\sigma= 
\lambda^{1/4}\bar{\mu}^{3/4}\sigma^{5/2}$.
We thus set
\[
\sigma = 
\left[4\SmN{1}^{-1}\left( \SmN{\qstar}/\SmN{1} \right)^{-\frac{3}{2\left(\qstar-1\right)}}\epsilon\right]^{2/5}.
\]
We can now substitute back
into our definitions $r=\bar{r}/\sigma$ and $\mu=\bar{\mu}\sigma^{2}/\lambda$
in Eq.~\eqref{eq:lb-mubar}
and verify that $r\ge1$ and $\mu\le1.$ For $r$, we have
\begin{equation*}
  r = \frac{\bar{r}}{\sigma} = 
  \left( \frac{\SmN{1}^{{\qstar}/{(\qstar-1)}}}{
    4\SmN{\qstar}^{{1}/{(\qstar-1)}}\epsilon}\right)^{2/5}
  = \left(
  \frac{\smCN{\qstar}}{
    4^{\qstar-1}\smCN{1}^{\qstar}}\right)^{\frac{2}{5(\qstar-1)}}
  \left(
  \frac{\Sm{1}^{{\qstar}/{(\qstar-1)}}}{
    \Sm{\qstar}^{{1}/{(\qstar-1)}}\epsilon}\right)^{2/5}
  \ge 1,
\end{equation*}
where the last transition uses $\smCN{\qstar}\ge
4^{\qstar-1}\smCN{1}^{\qstar}$ by the
definition~\eqref{eq:lb-maxlp}, and we used the assumption in
the theorem statement that $\epsilon\le\Sm
1^{\qstar/\left(\qstar-1\right)}/\Sm \qstar^{1/\left(\qstar-1\right)}$.
Similarly, our choice of $\mu$ satisfies
\begin{equation*}
  \mu=\frac{\bar{\mu}\sigma^{2}}{\lambda}=
  \left( \frac{4^{\qstar-1}\smCN{1}^{\qstar}}{
    \smCN{\qstar}}\right)^{\frac{4}{5(\qstar-1)}}
  \left( \frac{\Sm{\qstar}^{{1}/{(\qstar-1)}}\epsilon}{
    \Sm{1}^{{\qstar}/{(\qstar-1)}}}
  \right)^{4/5}
  \le 1.
\end{equation*}

We now consider two cases; $\lambda\sqrt{\mu}\sigma^2 \le \Delta$ and $\lambda\sqrt{\mu}\sigma^2 > \DeltaF$. In the first case (which holds for sufficiently small $\epsilon$), we substitute our choices of $\sigma,\lambda$ and $\mu$ into the time
lower bound~\eqref{eq:lb-genlb},
\begin{align*}
  T + 1 \ge \frac{\DeltaF-\lambda\sqrt{\mu}\sigma^{2}/2}{10\lambda\mu\sigma^{2}}
  \stackrel{\left(i\right)}{\ge} \frac{\DeltaF}{20\lambda\mu\sigma^{2}}
  & = \frac{\DeltaF}{20\bar{\mu}\sigma^{4}} 
  =
  \frac{\DeltaF\SmN{1}^\frac{3}{5}
    \left[\left(\SmN{\qstar}/\SmN{1}\right)^{\frac{1}{\qstar-1}}
      \right]^{\frac{2}{5}}
  }{20\cdot 4^{8/5}\cdot\epsilon^{8/5}} \\
  & =
  \frac{\DeltaF}{20 \cdot 4^{8/5}} \cdot
  \min_{q\ \in \{2, \ldots, p\} }
  \left\{
  \left(\SmGrad / \smCN{1}\right)^{\frac{3}{5}-\frac{2}{5(q-1)}}
  \left(\Sm{q} / \smCN{q}\right)^{\frac{2}{5(q-1)}}
  \right\} \epsilon^{-8/5},
\end{align*}
which is the desired bound, where in step~$(i)$ we made use of  $\lambda\sqrt{\mu}\sigma^2 \le \DeltaF$. When $\lambda\sqrt{\mu}\sigma^2 > \DeltaF$, we show that the above bound is in fact smaller than the convex lower bound in Theorem~\ref{thm:convex-lb}.  Indeed, substituting in our choices of $\lambda, \mu$ and $\sigma$, we see that $\lambda\sqrt{\mu}\sigma^2 > \DeltaF$ implies
\begin{equation*}
\Delta < 
 \SmN{1}^{3/2}  \left(\SmN{\qstar}/\SmN{1}\right)^{\frac{1}{\qstar-1}} \left[4\SmN{1}^{-1}\left( \SmN{\qstar}/\SmN{1} \right)^{-\frac{3}{2\left(\qstar-1\right)}}\epsilon\right]^{6/5}
= 
4^{\frac{6}{5}} \SmN{1}^{-\frac{1}{5}} \left(\SmN{\qstar}/\SmN{1}\right)^{-\frac{4}{5(\qstar-1)}}\epsilon^{\frac{6}{5}}.
\end{equation*}
Taking a square root and substituting to our lower bound gives
\begin{equation*}
 \frac{\DeltaF\SmN{1}^\frac{3}{5}
	\left[\left(\SmN{\qstar}/\SmN{1}\right)^{\frac{1}{\qstar-1}}
	\right]^{\frac{2}{5}}
}{20\cdot 4^{8/5}\cdot\epsilon^{8/5}} <
({80 \cdot \smCN{1}^{1/2}})^{-1} \frac{\sqrt{\DeltaF \SmGrad}}{\epsilon},
\end{equation*}
and therefore (recalling that $\QuadClass\subset\FclassMany{p}$ and that 
$\smCN{1} \ge 8$), by Thm.~\ref{thm:convex-lb} we have 
\begin{equation*}
	\CompEps{\AlgZRFO}{\FclassMany{p}} \ge \CompEps{\AlgZRFO}{\QuadClass} \ge 
 \frac{\sqrt{\DeltaF \SmGrad}}{4\epsilon} \ge
    \frac{\DeltaF\SmN{1}^\frac{3}{5}
  	\left[\left(\SmN{\qstar}/\SmN{1}\right)^{\frac{1}{\qstar-1}}
  	\right]^{\frac{2}{5}}
  }{20\cdot 4^{8/5}\cdot\epsilon^{8/5}},
\end{equation*}
completing the proof in the general case.

\paragraph{Functions with Lipschitz Hessian}
For $p=2$, we keep the definitions~\eqref{eq:lb-maxlp} but
replace the particular rescaling choices~\eqref{eq:lb-mubar} with
\begin{equation*}
  \lambda = \SmN{1} %
  ~,~
  \mu = \frac{\SmN{2}\sigma}{\lambda}
  ~~\mbox{and}~~
  r = 1.
\end{equation*}
Using $\mu\le1$, the above parameter setting satisfies
inequality~\eqref{eq:lb-lq};  $f$ has $\Sm{q}$-Lipschitz
$q$th-order derivatives for $q=1,2$.  To satisfy the gradient lower
bound~\eqref{eq:lb-eps}, \ie
$4\epsilon\le\lambda\mu^{3/4}\sigma=\SmN{1}^{1/4}\SmN{2}^{3/4}\sigma^{7/4}$,
we set
\begin{equation*}
  \sigma=\left[4\SmN{1}^{-1/4}\SmN{2}^{-3/4}\epsilon\right]^{4/7}.
\end{equation*}
We can substitute into the definition $\mu = \frac{\SmN{2} \sigma}{\lambda}$
to verify that $\mu\le1$:
\begin{align*}
  \mu & =\frac{\SmN{2}\sigma}{\lambda}=\left(\frac{4^{2-1}\left(\smCN 1\right)^{2}}{\smC 2}\right)^{4/7}\left(\frac{\Sm 2\epsilon}{\Sm 1^{2}}\right)^{4/7}\le1
\end{align*}
by the definition~\eqref{eq:lb-maxlp} of $\smCN{p}$ and the
assumption $\epsilon\le\Sm{1}^{2} / \Sm{2}$. As in the general case, we first assume $\lambda\sqrt{\mu}\sigma^2 \le \DeltaF$, where substituting
into~\eqref{eq:lb-genlb} yields the desired lower bound
\begin{align*}
  T+1\ge \frac{\DeltaF-\lambda\sqrt{\mu}\sigma^{2}/2}{10\lambda\mu\sigma^{2}}
  \ge %
  \frac{\DeltaF}{20\lambda\mu\sigma^{2}} = \frac{\DeltaF}{20\SmN{2}\sigma^{3}} 
  & =
  \frac{\DeltaF\SmN{1}^{\frac{3}{7}}\SmN{2}^{\frac{2}{7}}}{
    20\cdot 4^{12/7}\cdot\epsilon^{12/7}}.
\end{align*}
If $\lambda\sqrt{\mu}\sigma^2 \le \DeltaF$ does not hold, we have
\begin{equation*}
\Delta < 
 \SmN{1}^{1/2} \SmN{2}^{1/2} \left[4\SmN{1}^{-1/4}\SmN{2}^{-3/4}\epsilon\right]^{\frac{4}{7}\cdot \frac{5}{2}}
<
4^{\frac{10}{7}}\SmN{1}^{-\frac{1}{7}} \SmN{2}^{-\frac{2}{7}}\epsilon^{\frac{10}{7}}.
\end{equation*}
Taking a square root and substituting to our lower bound gives
\begin{equation*}
\frac{\DeltaF\SmN{1}^{\frac{3}{7}}\SmN{2}^{\frac{2}{7}}}{
	20\cdot 4^{12/7}\cdot\epsilon^{12/7}} <
({20 \cdot \smCN{1}^{1/2}})^{-1} \frac{\sqrt{\DeltaF \SmGrad}}{\epsilon} < 	\CompEps{\AlgZRFO}{\QuadClass} \le \CompEps{\AlgZRFO}{\FclassTwo{1}{2}},
\end{equation*}
due to Theorem~\ref{thm:convex-lb}, establishing the case $p=2$.

\section{The challenge of strengthening Theorem~\ref{thm:final}}
\label{sec:discussion}

The lower bounds in Theorem~\ref{thm:final} leave two avenues for
improvement. The first is tightening our $\epsilon^{-12/7}$ and
$\epsilon^{-8/5}$ lower bounds to match the known upper bounds of
$\epsilon^{-7/4}$ and $\epsilon^{-5/3}$, for $p=2$ and $p=3$,
respectively. The second improvement is to extend our lower bounds to
randomized algorithms, as we did for the case of full derivative information
in Section~\exref{sec:fullder-random}.
We discuss each of these in turn.

\subsection{Tightness of lower bound construction}
\label{sec:disc-tightness}

\newcommand{\fhardTFO}{\wt{f}_{T, \alpha, \mu}}
\newcommand{\linkfun}{\Lambda}
\newcommand{\bupsilon}{\wt{\Upsilon}}

The core of our first-order lower bounds is
Lemma~\ref{lem:gradbound}, which establishes a lower bound of the form
$\norms{\fhardFO(x)} > \mu^{3/4}/4$ for vectors $x$ such that
$x_\T=x_{\T+1}=0$ (\ie any point that a first-order zero-respecting method
can produce after $\T$ iterations), where $\fhardFO$ is our unscaled hard
instance (see definition~\eqref{eq:fhard-def}).  Here we consider a slightly 
more general form,
\begin{equation}\label{eq:gen-construction}
  \fhardTFO(x) \defeq  \alpha \cdot \linkfun (x_1 -1) 
  +\sum_{i=1}^{\T} \linkfun (x_{i+1} - x_i)
  + \mu \sum_{i=1}^{\T} \bupsilon(x_i),
\end{equation}
where $\linkfun : \R \to \R$ and $\bupsilon : \R \to \R$ are $\mc{C}^\infty$, 
and we assume $\T\in\N$, $\alpha > 0$,
and $0 < \mu\le 1$.
The chain $\fhardFO$
corresponds to the special case $\alpha = \sqrt{\mu}$, $\linkfun(x)=x^2/2$
and $\bupsilon(x) = \Ups(x)$ (defined in Eq.~\eqref{eq:ups-def}).

We
claim that if we can show that the norm of $\nabla \fhardTFO(x)$ is 
\emph{not too large} for some $x \in \R^{\T + 1}$ with $x_\T = x_{\T + 1}=0$, 
then 
our
lower bound cannot be improved. More concretely, suppose that for every
$\T$, $\mu\le1$ and $\alpha\le 1$ we could find ${x} \in \R^{\T+1}$ such
that $x_\T = x_{\T+1} =0$, and $\norms{\grad \fhardFO(x)} \le C\mu^{3/4}$
for some constant $C$ independent of $\T,r$ and $\mu$, matching
Lemma~\ref{lem:gradbound} to a constant. We can then trace the scaling
arguments in the proof of Theorem~\ref{thm:final} ``in reverse,'' showing
that any choice of $ \T,\lambda, \mu, \sigma$ and $r$ for which the function
$f(x) = \lambda \sigma^2 \fhardTFO(x/\sigma)$ satisfies both (i)
$f\in\FclassMany{p}$ and (ii) $\norm{\grad f(x)} > \epsilon$ for all $x$
such that $x_\T = x_{\T+1} = 0$, we have $T \le c \cdot \epsilon^{-8/5}$ for 
$p\ge
3$ and $T \le c \cdot \epsilon^{-12/7}$ for $p=2$, where $c$ is some
problem-dependent constant independent of $\epsilon$.

The next lemma, whose proof we provide in
Appendix~\ref{sec:proof-disc-tight}, shows such gradient norm upper bound 
 for constructions of the form~\eqref{eq:gen-construction}.
\begin{restatable}{lemma}{lemDistTight}\label{lem:disc-tight}
  Let $\T \in \N$, $0 < \alpha\le1$, $\mu \in [\T^{-2}, 1]$ and $\fhardTFO$ be defined as in~\eqref{eq:gen-construction}, with $\linkfun$ and $\bupsilon$ satisfying
  \begin{equation*}
    \linkfun'(0) = \bupsilon'(0) = 0 
    ~~\mbox{and}~~
    \linkfun'\mbox{ is 1-Lipschitz continuous}
    ~~\mbox{and}~~
    \max_{z\in [0,1]} | \bupsilon'(z) | \le G,
  \end{equation*}
  for $G>0$ independent of $\T, \alpha$ and $\mu$. Then there exists $x \in
  \R^{\T+1}$ such that $x_\T = x_{\T+1} = 0$ and
  \begin{equation*}
    \norm{\grad \fhardTFO(x)} < C\mu^{3/4},
  \end{equation*} 
  where $C \le 27 + \sqrt{3}G$.
\end{restatable}

Let us discuss the lemma.  The condition that
$\linkfun'(0) = \bupsilon'(0) = 0$ is essential for
any zero-chain-based proof, as otherwise $\fhardTFO$ is not a
first-order zero-chain (if $\alpha=1$ then we may 
have $\linkfun'(0)\neq 0$; Lemma~\ref{lem:disc-tight} holds in this case as
well).
The requirement that the multiplier $\mu \ge 1/T^2$ on $\bupsilon$
is also benign, as in
our proofs require $\mu \gtrsim 1/\sqrt{T} \gg 1/T^2$ (further decreasing 
$\mu$ weakens the lower bound as it makes $\fhardTFO$ too smooth; 
inspection of the scaling argument in the proof of Theorem~\ref{thm:final}  
shows this rigorously).
The function $\linkfun$ must have Lipschitz derivatives with parameter
independent of $\mu,\T$, as otherwise $\fhardTFO$ cannot be scaled to meet
the smoothness requirements. Finally, the requirement $ \max_{z\in [0,1]} |
\bupsilon'(z) | < \infty$ holds for every $\mc{C}^\infty$ function.
Moreover, a calculation shows it holds with $G=\sqrt{10\smC{3}}$ 
independent of $r$ for
 every $\Ups$ that satisfies Lemma~\ref{lem:props}.

Summarizing, tightening our lower bounds seems to require a construction
that is not of the form~\eqref{eq:gen-construction}. This does not eliminate
more general (non-convex) interactions, \eg of the form $\linkfun(x_i, x_{i + 
1})$ rather than
$\linkfun(x_{i+1} - x_i)$. The proof technique of Lemma~\ref{lem:gradbound} 
should provide useful ``sanity checks'' when considering alternative 
constructions.

\subsection{A bound for randomized algorithms}\label{sec:disc-rand}

In Section~\exref{sec:fullder-random} we extend our lower bound for
$\TimeEps{\AlgDet\cup\AlgZR}{\Fclass{p}}$ to the broader class of randomized
algorithms $\AlgRand$ with access to all derivatives at query point  
$x$. We do
this by making our hard function insensitive: the
individual ``linking'' terms $\compactfunc(x_i)\gausscdf(x_{i+1})$ are 
identically zero for $x_i$ near 0. A natural question is whether the
same methodology (originally proposed in~\cite{WoodworthSr16}) can extend
Theorem~\ref{thm:final} to the class of randomized first-order algorithms,
$\AlgRandFO$. Direct application of that technique cannot work in our case, 
for a simple reason: it applies to randomized algorithms of \emph{any}
order. In other words, if we modify our hard instance
construction~\eqref{eq:fhard-def} to be a robust zero-chain
(Definition~\exref{def:robust-zero-chain}), any lower bounds it implies hold
for all algorithms in $\AlgRand$, where $\epsilon^{-(p+1)/p}$ rates
are achievable, so we could not provide sharper lower bounds than
Theorem~\exref{thm:fullder-final}.

 Nevertheless, the ideas introduced in Section~\exref{sec:fullder-random} 
 might
 still be of use. Specifically, consider a modification of the
 construction~\eqref{eq:fhard-def} where $\Ups(x)$ is identically zero for
 sufficiently small $x$, say $|x| < 0.05$, while still satisfying %
 Lemma~\ref{lem:props}, thus making the
 non-convex component of $\fhardFO$ insensitive.  As
 explained above, also making the convex quadratic component of 
 $\fhardFO$ insensitive (as
 \citet{WoodworthSr16} do) is unworkable in our setting, as it results in a 
 robust zero-chain 
 equally hard for all high-order algorithms. Instead, we may keep the 
 quadratic component unchanged---and hence sensitive---and try to carry out 
 the proof of Lemma~\exref{lem:fullder-rand-slow}. Doing so,  we see that the 
 inductive argument allows us to ignore the insensitive non-convex 
 component of $\fhardFO$, leaving us to contend only with the (randomly 
 rotated) quadratic chain. Thus, the difficulty here appears closely related to 
 proving a lower bound for randomized first-order methods applied for 
 optimization of convex 
 quadratics. Such lower bounds remain elusive, and we believe finding them 
 is an important open problem.

\section{Concluding remarks}
\label{sec:concluding-remarks}

Here we situate our work in the literature, discuss its implications, and provide a few 
possible extensions.

\subsection{Commentary on our results}\label{sec:conclusion-implications}

In conjunction with known upper bounds, our lower bounds characterize the optimal rates for finding stationary points.
Our lower bounds are sharp to within constant factors for algorithms with full 
derivative information~\cite{NclbPartI},
and (perhaps) slightly loose for
first-order algorithms. These characterizations yield a few insights.

\paragraph{First-order methods vs.\ high-order methods}

For the class $\Fclass{1}$ of $\SmGrad$-smooth functions, first-order
methods---specifically gradient descent---attain the optimal rate
$\SmGrad\DeltaF\epsilon^{-2}$; no higher-order randomized method can 
attain
improved performance over the entire function class. The intuition
here is that $\Fclass{1}$ contains functions whose Hessian and higher
order derivatives may vary arbitrarily sharply, and providing no useful
information for optimization.

When higher-order derivatives are also Lipschitz continuous the picture
changes fundamentally: there is a strict separation between (deterministic)
second-order and first-order methods. In particular, cubic regularization of
Newton's method~\cite{NesterovPo06} achieves $\epsilon$ dependence
$\epsilon^{-3/2}$ for functions with Lipschitz Hessian, while \emph{no}
deterministic first-order method can have better time complexity than
$\epsilon^{-8/5}$, regardless of how many derivatives are Lipschitz.
Note that when the Hessian is Lipschitz, our
definition of first-order algorithms allows for algorithms that rely on
Hessian-vector products, as they can be estimated to arbitrary accuracy in
two gradient evaluations.

\paragraph{The effect of high-order smoothness on first-order 
methods}

For $\FclassTwo{1}{2}$, the class of functions with Lipschitz gradient and
Hessian, our lower bound scales as
$\epsilon^{-12/7}$, while for the class $\FclassThree{1}{2}{3}$ of functions
with Lipschitz third order derivative our ``convex until proven guilty''
method~\cite{CarmonDuHiSi17} achieves the rate
$\epsilon^{-5/3}\log\frac{1}{\epsilon}$. As $\frac{5}{3}<\frac{12}{7}$, this
proves a separation between the optimal rate for first-order methods with
second- and third-order smoothness.

In contrast, orders of smoothness beyond the third offer limited room for
improvement in $\epsilon$ dependence; the lower bound $\epsilon^{-8/5}$
holds for all function classes $\FclassMany{p}$ with $p\ge 3$, while the
method~\cite{CarmonDuHiSi17} does not enjoy improved guarantees with
Lipschitz fourth-order derivatives. The ``robustness'' of the lower bound to
higher-order smoothness stems from the fact that our hard instance
$\fhardFO$ becomes a quartic polynomial in the limit $r\to\infty$, and we
choose $r$ inversely proportional to $\epsilon$. As we discuss in
\cite[Lemma~4]{CarmonDuHiSi17}, our guarantee 
$\epsilon^{-5/3}\log\frac{1}{\epsilon}$ cannot
improve using fourth-order smoothness because of symmetries in the
fourth-order Taylor expansion.
Quartic polynomials thus appear to play a central role in the complexity
of first-order methods for smooth optimization.

\paragraph{Convex vs.\ non-convex functions}

Our results show that convexity makes finding stationary points
fundamentally---and significantly---easier. For first-order methods and
functions with bounded initial sub-optimality, the rate
$\epsilon^{-1}\log\frac{1}{\epsilon}$ is achievable for first-order smooth
convex functions, while the lower bound $\epsilon^{-8/5}$ holds for
non-convex functions with arbitrarily high-order smoothness.  For methods
using higher-order derivatives, our lower bounds~\cite{NclbPartI} for
finding stationary points of non-convex functions are $\epsilon^{-(p+1)/p} \to \epsilon^{-1}$ as the order $p$ of smoothness grows. However,
similar to Appendix~\ref{sec:app-convex-upper-bound},
the results~\cite{ArjevaniShSh17,MonteiroBe13} can show that
for convex functions with Lipschitz Hessian, a second-order method achieves
the strictly better rate $\epsilon^{-6/7}\log \frac{1}{\epsilon}$.

Another striking difference between convex and non-convex functions is the
effect of replacing the bound on the initial function value (\ie
$f(x\ind{0})-\inf_x f(x) \le \DeltaF$) with a bound on the initial distance
to the global minimizer $x\opt$ (\ie $\norm{x\ind{0}-x\opt}\le
D$). For non-convex function classes, we show lower bounds with the same
$\epsilon$ dependence regardless of which type of bound is used. In
contrast, for convex function the optimal rates scale as
$\sqrt{\DeltaF}\epsilon^{-1}$ and $\sqrt{D}\epsilon^{-1/2}$, again
a gap in $\epsilon$ dependence. The rates are not
directly comparable; one can construct families of functions where
$D$ grows with the dimension while $\DeltaF$ remains constant.

Returning to $\DeltaF$-value-bounded function classes, we see one more large
difference between the convex and non-convex case; convex rates scale as
$\sqrt{\DeltaF}$ while all the non-convex rates scale linearly with
$\DeltaF$. This arises from fundamental differences in the
convergence ``mechanism'' for convex and non-convex optimization. The
analysis of non-convex optimization schemes typically~\cite{Nesterov04,
  NesterovPo06, CartisGoTo10, Nesterov12b, CarmonDuHiSi17} revolves around a
\emph{progress argument}, where one shows that, as long as $\norms{\grad f(x\ind{t})}>\epsilon$, the guarantee $f(x\ind{t + 1}) \le
f(x\ind{t}) - p_\epsilon$ holds for some quantity $p_\epsilon$ (\eg for gradient 
descent $p_\epsilon=\epsilon^2/(2\SmGrad)$). The number of iterations to 
find an $\epsilon$-stationary point $x_s$ is therefore at most 
$[f(x\ind{0})-f(x_s)]/p_\epsilon\le \DeltaF/p_\epsilon$, which scales linearly in 
$\DeltaF$.  By our
lower bounds, such progress arguments are, in a sense, optimal.
Conversely, in convex optimization we
may control either the gap $f(x\ind{t}) - f(x\opt)$ or the distance
$\norm{x\ind{t} - x\opt}$, and this interplay (see
Appendix~\ref{sec:app-convex-upper-bound}) allows stronger 
arguments than those  based purely on function progress.

\subsection{Further research}\label{sec:conclusion-future}

\paragraph{Closing the gap in first-order bounds}  

There exists a gap in polynomial $\epsilon$ dependence between our
lower bounds (Theorem~\ref{thm:final}) and the best known upper
bounds~\cite{CarmonDuHiSi17} for first-order methods with higher-order
smoothness.
We do not believe the upper bounds of~\cite{CarmonDuHiSi17} are improvable
by different analysis or by any algorithmic change that maintains the
general structure of alternating between accelerated gradient descent and
negative curvature exploitation. In conjunction with our arguments in
Section~\ref{sec:disc-tightness} about the structure of our lower bounds,
resolution of the optimal rate will likely provide either a method with a
substantially different approach to accelerating gradient descent in the
smooth non-convex setting or a new lower bound construction. 

\paragraph{Finite sum and stochastic problems}
Smooth, non-convex, finite-sum and stochastic optimization problems are
important, arising (for example) in the training of neural
networks. This motivates the design and analysis of efficient methods for
finding stationary points in such problems, and researchers have
successfully developed variance reduction and acceleration techniques for
these settings~\cite{ReddiHeSrPoSm16,AllenHa16,LeiJuChJo17,Allen17b}.
However, no
corresponding lower bounds are available. 
\citet{WoodworthSr16}, show how
to establish lower bounds for \emph{convex} finite sum problems. Combined
with the developments in our paper, we believe their techniques should
extend to finding stationary points of non-convex problems. An important
conclusion of~\cite{WoodworthSr16} is that randomized selection of the
component function is crucial to efficient convergence: in contrast to our
results, they show a separation between deterministic and randomized finite
sum complexity.

\paragraph{Second-order stationary points}
Approximate stationary points are not always close to local minima, and so it is interesting to consider stronger convergence guarantees. Second-order stationarity (also known as the second-order necessary
condition for local optimality) is the most popular example; for a function
$f$, a point $x$ is $(\epsilon_1, \epsilon_2)$-second-order stationary if
$\norm{\grad f(x)} \le \epsilon_1$ and $\deriv{2}f(x) \succeq -\epsilon_2
I$.  Efficient first-order methods for finding second-order stationary points
exist~\cite{AgarwalAlBuHaMa17,CarmonDuHiSi16,JinGeNeKaJo17}. Moreover, it is
possible to generically transform methods for finding $\epsilon$-stationary
points into methods that find $(\epsilon, O(\epsilon^s))$-second-order
stationary points, for some $0<s<1$, without changing the $\epsilon$
dependence of the complexity~\cite[Appendix~C]{CarmonDuHiSi17},
but such modifications introduce dependence logarithmic in the problem 
dimension $d$.
 
Clearly, lower bounds for finding $\epsilon_1$-stationary points also
apply to finding $(\epsilon_1, \epsilon_2)$-second-order stationary
points. However, attaining second-order stationarity with {first-order} methods
is fundamentally more difficult than attaining only stationarity.
There are no dimension-free guarantees: the results
of~\citet{SimchowitzAlRe17} imply $\Omega(\log d)$ dimension dependence for
all randomized first-order algorithms that escape saddle points. Moreover,
for deterministic first-order algorithms it is easy to construct a resisting
oracle that forces $\Omega(d)$ dimension dependence (consider $f(R
x)$ with $f(x) = -x_{1}^2$ and adversarially chosen rotation $R$), implying
strong separation between deterministic and randomized first-order methods
for finding second-order stationary points. It will
be interesting to investigate such issues further.

\section*{Acknowledgments}
OH was supported by the PACCAR INC fellowship. YC and JCD were partially
supported by the SAIL-Toyota Center for AI Research, NSF-CAREER award
1553086, and a Sloan Foundation Fellowship in Mathematics. YC was partially
supported by the Stanford Graduate Fellowship and the Numerical Technologies
Fellowship.

\bibliographystyle{abbrvnat}
\bibliography{bib,nclb}

\begin{thebibliography}{24}
\providecommand{\natexlab}[1]{#1}
\providecommand{\url}[1]{\texttt{#1}}
\expandafter\ifx\csname urlstyle\endcsname\relax
  \providecommand{\doi}[1]{doi: #1}\else
  \providecommand{\doi}{doi: \begingroup \urlstyle{rm}\Url}\fi

\bibitem[Agarwal et~al.(2017)Agarwal, Allen-Zhu, Bullins, Hazan, and
  Ma]{AgarwalAlBuHaMa17}
N.~Agarwal, Z.~Allen-Zhu, B.~Bullins, E.~Hazan, and T.~Ma.
\newblock Finding approximate local minima faster than gradient descent.
\newblock In \emph{Proceedings of the Forty-Ninth Annual ACM Symposium on the
  Theory of Computing}, 2017.
\newblock URL \url{https://arxiv.org/abs/1611.01146}.

\bibitem[Allen-Zhu(2017)]{Allen17b}
Z.~Allen-Zhu.
\newblock Natasha 2: Faster non-convex optimization than {SGD}.
\newblock \emph{arXiv:1708.08694 [math.OC]}, 2017.

\bibitem[Allen-Zhu and Hazan(2016)]{AllenHa16}
Z.~Allen-Zhu and E.~Hazan.
\newblock Variance reduction for faster non-convex optimization.
\newblock In \emph{Proceedings of the 32nd International Conference on Machine
  Learning}, 2016.

\bibitem[Arjevani et~al.(2017)Arjevani, Shamir, and Shiff]{ArjevaniShSh17}
Y.~Arjevani, O.~Shamir, and R.~Shiff.
\newblock Oracle complexity of second-order methods for smooth convex
  optimization.
\newblock \emph{arXiv:1705.07260 [math.OC]}, 2017.

\bibitem[Birgin et~al.(2017)Birgin, Gardenghi, Mart\'inez, Santos, and
  Toint]{BirginGaMaSaTo17}
E.~G. Birgin, J.~L. Gardenghi, J.~M. Mart\'inez, S.~A. Santos, and P.~L. Toint.
\newblock Worst-case evaluation complexity for unconstrained nonlinear
  optimization using high-order regularized models.
\newblock \emph{Mathematical Programming}, 163\penalty0 (1--2):\penalty0
  359--368, 2017.

\bibitem[Boyd and Vandenberghe(2004)]{BoydVa04}
S.~Boyd and L.~Vandenberghe.
\newblock \emph{Convex Optimization}.
\newblock Cambridge University Press, 2004.

\bibitem[Boyd et~al.(2011)Boyd, Parikh, Chu, Peleato, and
  Eckstein]{BoydPaChPeEc11}
S.~Boyd, N.~Parikh, E.~Chu, B.~Peleato, and J.~Eckstein.
\newblock Distributed optimization and statistical learning via the alternating
  direction method of multipliers.
\newblock \emph{Foundations and Trends in Machine Learning}, 3\penalty0 (1),
  2011.

\bibitem[Carmon et~al.(2016)Carmon, Duchi, Hinder, and Sidford]{CarmonDuHiSi16}
Y.~Carmon, J.~C. Duchi, O.~Hinder, and A.~Sidford.
\newblock Accelerated methods for non-convex optimization.
\newblock \emph{arXiv:1611.00756 [math.OC]}, 2016.

\bibitem[Carmon et~al.(2017{\natexlab{a}})Carmon, Duchi, Hinder, and
  Sidford]{CarmonDuHiSi17}
Y.~Carmon, J.~C. Duchi, O.~Hinder, and A.~Sidford.
\newblock Convex until proven guilty: dimension-free acceleration of gradient
  descent on non-convex functions.
\newblock In \emph{Proceedings of the 33rd International Conference on Machine
  Learning}, 2017{\natexlab{a}}.

\bibitem[Carmon et~al.(2017{\natexlab{b}})Carmon, Duchi, Hinder, and
  Sidford]{NclbPartI}
Y.~Carmon, J.~C. Duchi, O.~Hinder, and A.~Sidford.
\newblock {\hightarget{LinkPartI}{L}}ower bounds for finding stationary points
  {I}.
\newblock \emph{arXiv: 1710.11606 [math.OC]}, 2017{\natexlab{b}}.
\newblock URL \url{https://arxiv.org/pdf/1710.11606.pdf}.

\bibitem[Cartis et~al.(2010)Cartis, Gould, and Toint]{CartisGoTo10}
C.~Cartis, N.~I. Gould, and P.~L. Toint.
\newblock On the complexity of steepest descent, {N}ewton's and regularized
  {N}ewton's methods for nonconvex unconstrained optimization problems.
\newblock \emph{SIAM Journal on Optimization}, 20\penalty0 (6):\penalty0
  2833--2852, 2010.

\bibitem[Chowla et~al.(1951)Chowla, Herstein, and Moore]{ChowlaHeMo51}
S.~Chowla, I.~N. Herstein, and W.~K. Moore.
\newblock On recursions connected with symmetric groups {I}.
\newblock \emph{Canadian Journal of Mathematics}, 3:\penalty0 328--334, 1951.

\bibitem[Chung(1998)]{Chung98}
F.~R.~K. Chung.
\newblock \emph{Spectral Graph Theory}.
\newblock AMS, 1998.

\bibitem[Jin et~al.(2017)Jin, Ge, Netrapalli, Kakade, and
  Jordan]{JinGeNeKaJo17}
C.~Jin, R.~Ge, P.~Netrapalli, S.~M. Kakade, and M.~I. Jordan.
\newblock How to escape saddle points efficiently.
\newblock In \emph{Proceedings of the 33rd International Conference on Machine
  Learning}, 2017.

\bibitem[Lei et~al.(2017)Lei, Ju, Chen, and Jordan]{LeiJuChJo17}
L.~Lei, C.~Ju, J.~Chen, and M.~I. Jordan.
\newblock Nonconvex finite-sum optimization via {SCSG} methods.
\newblock \emph{arXiv:1706.09156 [math.OC]}, 2017.

\bibitem[Monteiro and Svaiter(2013)]{MonteiroBe13}
R.~D. Monteiro and B.~F. Svaiter.
\newblock An accelerated hybrid proximal extragradient method for convex
  optimization and its implications to second-order methods.
\newblock \emph{SIAM Journal on Optimization}, 23\penalty0 (2):\penalty0
  1092--1125, 2013.

\bibitem[Nemirovski and Yudin(1983)]{NemirovskiYu83}
A.~Nemirovski and D.~Yudin.
\newblock \emph{Problem Complexity and Method Efficiency in Optimization}.
\newblock Wiley, 1983.

\bibitem[Nesterov(2004)]{Nesterov04}
Y.~Nesterov.
\newblock \emph{Introductory Lectures on Convex Optimization}.
\newblock Kluwer Academic Publishers, 2004.

\bibitem[Nesterov(2012)]{Nesterov12b}
Y.~Nesterov.
\newblock How to make the gradients small.
\newblock \emph{Optima}, 88, 2012.

\bibitem[Nesterov and Polyak(2006)]{NesterovPo06}
Y.~Nesterov and B.~Polyak.
\newblock Cubic regularization of {N}ewton method and its global performance.
\newblock \emph{Mathematical Programming, Series A}, 108:\penalty0 177--205,
  2006.

\bibitem[Reddi et~al.(2016)Reddi, Hefny, Sra, Poczos, and
  Smola]{ReddiHeSrPoSm16}
S.~J. Reddi, A.~Hefny, S.~Sra, B.~Poczos, and A.~Smola.
\newblock Stochastic variance reduction for nonconvex optimization.
\newblock In \emph{Proceedings of the 32nd International Conference on Machine
  Learning}, 2016.

\bibitem[Simchowitz et~al.(2017)Simchowitz, Alaoui, and
  Recht]{SimchowitzAlRe17}
M.~Simchowitz, A.~E. Alaoui, and B.~Recht.
\newblock On the gap between strict-saddles and true convexity: An
  {$\Omega(\log d)$} lower bound for eigenvector approximation.
\newblock \emph{arXiv:1704.04548 [cs.LG]}, 2017.

\bibitem[Woodworth and Srebro(2016)]{WoodworthSr16}
B.~E. Woodworth and N.~Srebro.
\newblock Tight complexity bounds for optimizing composite objectives.
\newblock In \emph{Advances in Neural Information Processing Systems 29}, pages
  3639--3647, 2016.

\bibitem[Zhang et~al.(2012)Zhang, Ling, and Qi]{ZhangLiQi12}
X.~Zhang, C.~Ling, and L.~Qi.
\newblock The best rank-1 approximation of a symmetric tensor and related
  spherical optimization problems.
\newblock \emph{SIAM Journal on Matrix Analysis and Applications}, 33\penalty0
  (3):\penalty0 806--821, 2012.

\end{thebibliography}

\newpage
\appendix
\section{Additional results for convex functions}\label{sec:app-convex}

\subsection{An upper bound for finding stationary points of value-bounded functions}\label{sec:app-convex-upper-bound}

Here we give a first-order method that finds $\epsilon$-stationary points of a 
function $f\in\ConvClass$ in $O(\sqrt{\SmGrad 
\DeltaF}\epsilon^{-1}\log\frac{\SmGrad\DeltaF}{\epsilon^2})$ iterations. 
The method consists of Nesterov's accelerated gradient descent (AGD) 
applied on the sum of $f$ and a standard quadratic regularizer.

Our starting point is AGD for
\emph{strongly convex} functions; a function $f$ is $\StrConv$-strongly
convex if
\begin{equation*}
f(y) \ge f(x) + \inner{\grad f(x)}{y-x} + \frac{\StrConv}{2}\norm{y-x}^2,
\end{equation*}
for every $x,y$ in the domain of $f$. Let $\agd_{\StrConv, \SmGrad}\in \AlgZRFO\cap \AlgDetFO$ be the accelerated gradient scheme developed in \cite[\S 2.2.1]{Nesterov04} for $\sigma$-strongly convex functions with $\SmGrad$-Lipschitz gradient, initialized at $x\ind{1}=0$ (the exact step size scheme is not important). Adapting \cite[Thm. 2.2.2]{Nesterov04} to our notation, for any  $\StrConv$-strongly-convex $f$ with $\SmGrad$-Lipschitz gradient and (unique) global minimizer $x^\star_f$, the iterates $\{x\ind{t}\}_{t\in\N} = \agd_{\StrConv, \SmGrad}[f]$ of AGD satisfy
\begin{equation}\label{eq:agd-guarantee-base}
f(x\ind{t}) - f(x_f^\star) \le (1-\sqrt{\StrConv/\SmGrad})^{t-1} \SmGrad\norm{x_f^\star}^2
~~\mbox{and}~~
f(x\ind{t}) \le f(0)
~~\mbox{for every}~t\in\N.
\end{equation}
Moreover, for any such $f$ we have $\norm{\grad f(x)}^2 \le 2\SmGrad(f(x)-f(x^\star_f))$ \cite[Eq. (9.14)]{BoydVa04}, and consequently
\begin{equation*}
\norms{\grad f(x\ind{t})} \le (1-\sqrt{\StrConv/\SmGrad})^{(t-1)/2} \SmGrad\norm{x_f^\star}~~\mbox{for every}~t\in\N.
\end{equation*} 
Using our complexity notation, we may rewrite this as
\begin{equation}\label{eq:agd-guarantee}
\TimeEps{\agd_{\StrConv, \SmGrad}}{f} \le 
1+ 2\sqrt{\frac{\SmGrad}{\sigma}}\log_+\left(\frac{\SmGrad \norms{x_f^\star}}{\epsilon}\right),
\end{equation}
with $\log_+(x) \defeq \max\{0, \log x\}$. 

Now suppose that $f$ is convex with $\SmGrad$-Lipschitz gradient but not 
necessarily strongly-convex. We can add strong convexity to $f$ by means of 
a proximal term; for any $\sigma>0$, the function
\begin{equation*}
f_\StrConv(x)\defeq f(x) + \frac{\StrConv}{2}\norm{x}^2
\end{equation*}
is $\StrConv$-strongly-convex with $(\SmGrad+\StrConv)$-Lipschitz gradient. With this in mind, we define a proximal version of AGD as follows,
\begin{equation*}
\pagd_{\StrConv, \SmGrad}[f] \defeq \agd_{\StrConv, \SmGrad+\StrConv}[f_\StrConv] = \agd_{\StrConv, \SmGrad+\StrConv}\left[f(\cdot) + \frac{\StrConv}{2}\norm{\cdot}^2\right].
\end{equation*}
With a careful choice of $\sigma$, $\pagd_{\StrConv, \SmGrad}$ achieves the desired upper bound.

\begin{proposition}\label{prop:complexity-stationary-convex}
	Let $\DeltaF,\SmGrad$ and $\epsilon$ be positive, and let $\StrConv = \frac{\epsilon^2}{3\DeltaF}$. Then, algorithm $\pagd_{\StrConv, \SmGrad} \in \AlgDetFO$ satisfies
	\begin{equation*}
	\CompEps{\AlgDetFO}{\ConvClass} \le 
	\sup_{f\in\ConvClass} \TimeEps{\pagd_{\StrConv, \SmGrad}}{f} \le 1+5\frac{\sqrt{\SmGrad\DeltaF}}{\epsilon}\log_+ \left( \frac{25\SmGrad\DeltaF}{\epsilon^2}\right).
	\end{equation*}
\end{proposition}

\begin{proof}
  For any $f\in\ConvClass$, recall that $f_\StrConv(x) \defeq f(x)
  +\frac{\sigma}{2}\norm{x}^2$ and let
  \begin{equation*}
    \{x\ind{t}\}_{t\in\N} = \pagd_{\StrConv, \SmGrad}[f] = \agd_{\StrConv, \SmGrad+\StrConv}[f_\sigma]
  \end{equation*}
  be the sequence of iterates $\pagd_{\StrConv, \SmGrad}$ produces on
  $f$. Then by guarantee~\eqref{eq:agd-guarantee}, we have
  \begin{equation}\label{eq:app-convex-proxeps}
    \norm{\grad f_\sigma(x\ind{T})} \le \epsilon/6
  \end{equation}
  for some $T$ such that
  \begin{equation}\label{eq:app-convex-tdef}
    T \le  
    1+ 2\sqrt{1+\frac{\SmGrad}{\sigma}}\log_+\left(\frac{6(\SmGrad+\StrConv) \norms{x_{f_\sigma}^\star}}{\epsilon}\right).
  \end{equation}
  For any point $y$ such that $f_\StrConv(y) = f(y) + \frac{\sigma}{2}
  \norm{y}^2 \le f_\StrConv(0)=f(0)$, we have
  \begin{equation*}
    \norm{y}^2 \le \frac{2(f(0)-f(y))}{\sigma} \le \frac{2(f(0)-\inf_x f(x))}{\sigma} \le \frac{2\Delta}{\StrConv}.
  \end{equation*}
  Clearly, $f_\sigma(x_{f_\sigma}^\star) \le f_\sigma(0)$ and by
  guarantee~\eqref{eq:agd-guarantee-base} we also have $f_\sigma(x\ind{T})
  \le f_\sigma(0)$. Consequently,
  \begin{equation}\label{eq:app-convex-normeps}
    \max\left\{\norms{x\ind{T}}, \norms{x_{f_\sigma}^\star}
    \right\}\le \sqrt{\frac{2\DeltaF}{\StrConv}},
  \end{equation}
  and so
  \begin{equation*}
    \norms{\grad f(x\ind{T})}
    =\norms{\grad f_\sigma(x\ind{T})-\sigma \cdot x\ind{T}}
    \le  \norms{\grad f_\sigma(x\ind{T})} + \sigma \norms{x\ind{T}}
    \stackrel{(i)}{\le} \frac{\epsilon}{6} + \sqrt{2\StrConv\DeltaF}
    \stackrel{(ii)}{\le} \epsilon.
  \end{equation*}
  In inequality $(i)$ we substituted bounds~\eqref{eq:app-convex-proxeps}
  and~\eqref{eq:app-convex-normeps}, and in $(ii)$ we used $\sigma =
  \epsilon^2/(3\Delta)$. We conclude that
  $\TimeEps{\pagd_{\StrConv, \SmGrad}}{f} \le T$, and
  substituting~\eqref{eq:app-convex-normeps} and the definition of $\sigma$
  into~\eqref{eq:app-convex-tdef} we have
  \begin{equation*}
    T \le  
    1+ 2\sqrt{1+\frac{3\SmGrad\DeltaF}{\epsilon^2}}\log_+\left(6\sqrt{\frac{2}{3}}+\frac{6\sqrt{6}\SmGrad \DeltaF}{\epsilon^2}\right).
  \end{equation*}
  Without loss of generality, we may assume $\frac{2 \SmGrad
    \DeltaF}{\epsilon^2} \ge 1$, as otherwise
  $\TimeEps{\pagd_{\StrConv, \SmGrad}}{f} = 1$.
  We thus simplify the expression slightly to obtain the proposition.
\end{proof}

\subsection{The impossibility of approximate optimality without a 
bounded domain}
\label{sec:convex-takes-forever}

\begin{lemma}
  \label{lemma:convex-takes-forever}
  Let $\SmGrad,\DeltaF>0$ and $\epsilon < \DeltaF$. For any first-order
  algorithm $\alg\in\AlgDetFO \cup \AlgZRFO$ and any $\T\in\N$, there exists
  a function $f\in\QuadClass$ such that the iterates $\{x\ind{t}\}_{t \in
    \N} = \alg[f]$ satisfy
  \begin{equation*}
    \inf_{t \in \N} \left\{t \mid f(x\ind{t}) \le \inf_x f(x) + \epsilon
    \right\} > T.
  \end{equation*}
\end{lemma}

\begin{proof}
  By Proposition~\ref{prop:recap-det-zr} it suffices to consider $\alg\in
  \AlgZRFO$ (see additional discussion of the generality of
  Proposition~\ref{prop:recap-det-zr} in Section 
  \exref{sec:anatomy-reduction}). Consider the function $f:\R^\T\to\R$,
  \begin{equation}
    f(x) = \lambda
    \left[(\sigma - \beta x_1)^2
      + \sum_{i = 1}^{\T-1} (x_i - \beta x_{i + 1})^2
      \right],
    \label{eqn:convex-takes-forever}
  \end{equation}
  where $0<\beta < 1$, and we take
  \begin{equation*}
    \lambda \defeq \frac{\SmGrad}{2(1+2\beta+\beta^2)} 
    ~~\mbox{and}~~
    \sigma \defeq \sqrt{\frac{\DeltaF}{\lambda}}. 
  \end{equation*}
  Since $f(x)$ is of the form $\lambda\norm{A x-b}^2$ where 
  $\opnorm{A}\le 1+\beta$, we have $\opnorm{\deriv{2}f(x)}\le 
  2\lambda \opnorm{A}^2$ for every $x\in\R^\T$ and therefore $f$ has $2 
  \lambda(1 + 2 
  \beta + \beta^2)$-Lipschitz gradient. Additionally, $f$ 
  satisfies $\inf_x f(x) = 0$ and $f(0) = \lambda \sigma^2$, ans so the above
  choices of $\lambda$ and $\sigma$ guarantee that $f\in \QuadClass$.
  Moreover, $f$ is a a first-order zero-chain
  (Definition~\ref{def:recap-zero-chain}), and thus for any
  $\alg\in\AlgZRFO$ and $\{x\ind{t}\}_{t \in \N} = \alg[f]$, we have
  $x\ind{t}_{\T} = 0$ for $t\le \T$
  (Observation~\ref{obs:recap-zero-chain}). Therefore, it suffices to show that $f(x) >
  \inf_y f(y) + \epsilon$ whenever $x_T =0$.

  We make the following inductive claim: if
  $f(x) \le \inf_y f(y) + \epsilon = \epsilon$, then
  \begin{equation} \label{eqn:induction-sucky-deltaf}
	\left| x_i - \sigma\beta^{-i} \right| \le  \sum_{j=1}^{i} \beta^{-j} \sqrt{\frac{\epsilon}{\lambda}} < \frac{\beta^{-i}}{1-\beta}\sqrt{\frac{\epsilon}{\lambda}}
  \end{equation}
  for all $i\le\T$. Indeed, each term in the sum~\eqref{eqn:convex-takes-forever} defining
  $f$ is non-negative, so for the base case of the induction $i=1$,
  we have $\lambda(\sigma - \beta x_1)^2 \le \epsilon$, or
  $\left| x_1 - \sigma \beta^{-1}\right| \le \beta^{-1} \sqrt{\epsilon / \lambda}$. For $i<\T$, assuming
  that $x_i$ satisfies the bound~\eqref{eqn:induction-sucky-deltaf},
  we have that $\lambda (x_i - \beta x_{i + 1})^2 \le \epsilon$, which implies
  \begin{flalign*}
    \left | x_{i + 1}-\sigma\beta^{-(i+1)} \right| & \le 
    	\left | x_{i + 1}-\beta^{-1}x_i \right| 
    	+ \beta^{-1}\left | x_i-\sigma\beta^{-i} \right| %
    \le \beta^{-1}\sqrt{\frac{\epsilon}{\lambda}} 
    + \sum_{j=1}^{i} \beta^{-(j+1)} \sqrt{\frac{\epsilon}{\lambda}} =  \sum_{j=1}^{i+1} \beta^{-j} \sqrt{\frac{\epsilon}{\lambda}}, 
  \end{flalign*}
  which is the desired claim~\eqref{eqn:induction-sucky-deltaf} for $x_{i+1}$. 
  
  The bound~\eqref{eqn:induction-sucky-deltaf} implies $x_i\ne 0$ for all $i\le\T$ whenever $\sigma \ge (1-\beta)^{-1}\sqrt{\epsilon/\lambda}$. Therefore, we choose $\beta$ to satisfy $\sigma = (1-\beta)^{-1}\sqrt{\epsilon/\lambda}$, that is
  \begin{equation*}
  \beta \defeq 1- \sqrt{\frac{\epsilon}{\lambda \sigma^2}} = 1- \sqrt{\frac{\epsilon}{\DeltaF}},
  \end{equation*}
  for which $0<\beta<1$ since we assume $\epsilon<\DeltaF$. Thus, we guarantee that when $x_T=0$ we must have $f(x) > \inf_y f(y) + \epsilon$, giving the result.
\end{proof}

\section{Technical results}
\label{sec:proofs}

\subsection{Proof of Lemma~\ref{lem:props}}\label{sec:props-proof}

\lemUpsProps*

\begin{proof}
	Parts~\ref{item:ups-grad-zero} and~\ref{item:ups-quasi-convex}
	are evident from inspection, as
	\begin{equation*}
	\Ups'(x) = 120 \frac{x^2 (x - 1)}{1 + (x/r)^2}.
	\end{equation*}
	To see the part~\ref{item:ups-min-one}, note that $\Ups$ is non-increasing
	for every $x<1$ and non-decreasing for every $x>1$ and therefore $x=1$ is
	its global minimum. That $\Ups(1)=0$ is immediate from its definition,
	and, for every $r$, $\Ups(0) = 120\int_0^1 \frac{t^2(1-t)}{1+(t/r)^2}dt
	\le 120\int_0^1 t^2(1-t)dt = 10$. To see part~\ref{item:ups-large-grad},
	note that $|\Ups'(x)| \ge |\Upsilon_1'(x)|$ for every $r \ge 1$, and a
	calculation shows $|\Upsilon_1'(x)| > 1$ for $x\in \openleft{-\infty}{-0.1}
	\cup [0.1,0.9]$ (see Figure~\ref{fig:construction}).
	
	To see the fifth part of the claim, note that
	\begin{equation*}
	\Ups'(x) = 120r^2(x-1)\left(  1 - \frac{1}{1+(x/r)^2}\right)
	= 120\left[r^2 (x-1) - r^3 \varphi_1(x/r) + r^2 \varphi_2(x/r)\right],
	\end{equation*}
	where the functions $\varphi_1$ and $\varphi_2$ are $\varphi_1(\xi) =
	\xi/(1+\xi^2)$ and $\varphi_2(\xi) = 1/(1+\xi^2)$.  We thus bound the
	derivatives of $\varphi_1$ and $\varphi_2$.  We begin with $\varphi_2$,
	which we can write as the composition $\varphi_2(x) = (h \circ g)(x)$
	where $h(x) = \frac{1}{x}$ and $g(x) = 1 + x^2$.
	Let $\mc{P}_{k,2}$ denote the collection of all partitions of
	$\{1, \ldots, k\}$ where each element of the partition has at most
	$2$ indices. That is, if $P \in \mc{P}_{k,2}$, then
	$P = (S_1, \ldots, S_l)$ for some $l \le k$, the $S_i$ are disjoint,
	$1 \le |S_i| \le 2$, and $\cup_i S_i = [k]$.
	The cardinality $|\mc{P}_{k,2}|$ is the number of matchings in the complete
	graph on $k$ vertices, or the $k$th telephone number,
	which has bound~\cite[Lemma 2]{ChowlaHeMo51}
	\begin{equation*}
	|\mc{P}_{k,2}| \le \exp\left(\frac{k}{2} \log k + k \log 2\right).
	\end{equation*}
	We may then apply Fa\`{a} di Bruno's formula for the chain rule to obtain
	\begin{equation*}
	\varphi_2^{(k)}(x)
	= \sum_{P \in \mc{P}_k}
	h^{(|P|)}(g(x))
	\prod_{S \in P} g^{(|S|)}(x)
	= \sum_{P \in \mc{P}_{k,2}}
	(-1)^{|P|} \frac{(|P| - 1)!}{(1 + x^2)^{|P|}}
	(2x)^{\countset_1(P)} 2^{\countset_2(P)},
	\end{equation*}
	where $\countset_i(P)$ denotes the number of sets in $P$ with precisely
	$i$ elements.  Of course, we have $|x|^{\countset_1(P)} / (1 + x^2)^{|P|}
	\le 1$, and thus 
	\begin{equation*}
	|\varphi_2^{(k)}(x)|
	\le \sum_{P \in \mc{P}_{k,2}}
	(|P| - 1)! 2^{|P|}
	\le |\mc{P}_{k,2}| \cdot  (k - 1)!  \cdot 2^{k}
	\le e^{\frac{3k}{2} \log k + 2k \log 2}.
	\end{equation*}
	The proof of the upper bound on
	$\varphi_1^{(k)}(x)$ is similar ($2\varphi_1(x)=\frac{d}{dx}[(\hat{h}\circ 
	g)(x)]$ with 
	$\hat h(x) = \log x$ and $g$ as defined above), so 
	for every $r \ge 1$ and $p \ge 1$,
	the $p+1$-th derivative of $\Ups$ has the bound
	\begin{align*}
	|\Ups^{(p+1)}(x)|
	& \le 120\left[r^2 \indic{p = 1}
	+ r^{3-p}|\varphi_1^{(p)}(x/r)| + r^{2-p}|\varphi_2^{(p)}(x/r)|
	\right] 
	\le 120 r^{3 - p}e^{\frac{3}{2}\log p  + cp},
	\end{align*}
	where $c < \infty$ is a numerical constant. 
\end{proof}

\subsection{Proof of Lemma~\ref{lem:gradbound}}
\label{sec:gradbound-proof}

\lemFirstorderGradbound*

\begin{proof}
  \begin{figure}
	\centering
	\includegraphics[width=0.48\textwidth]{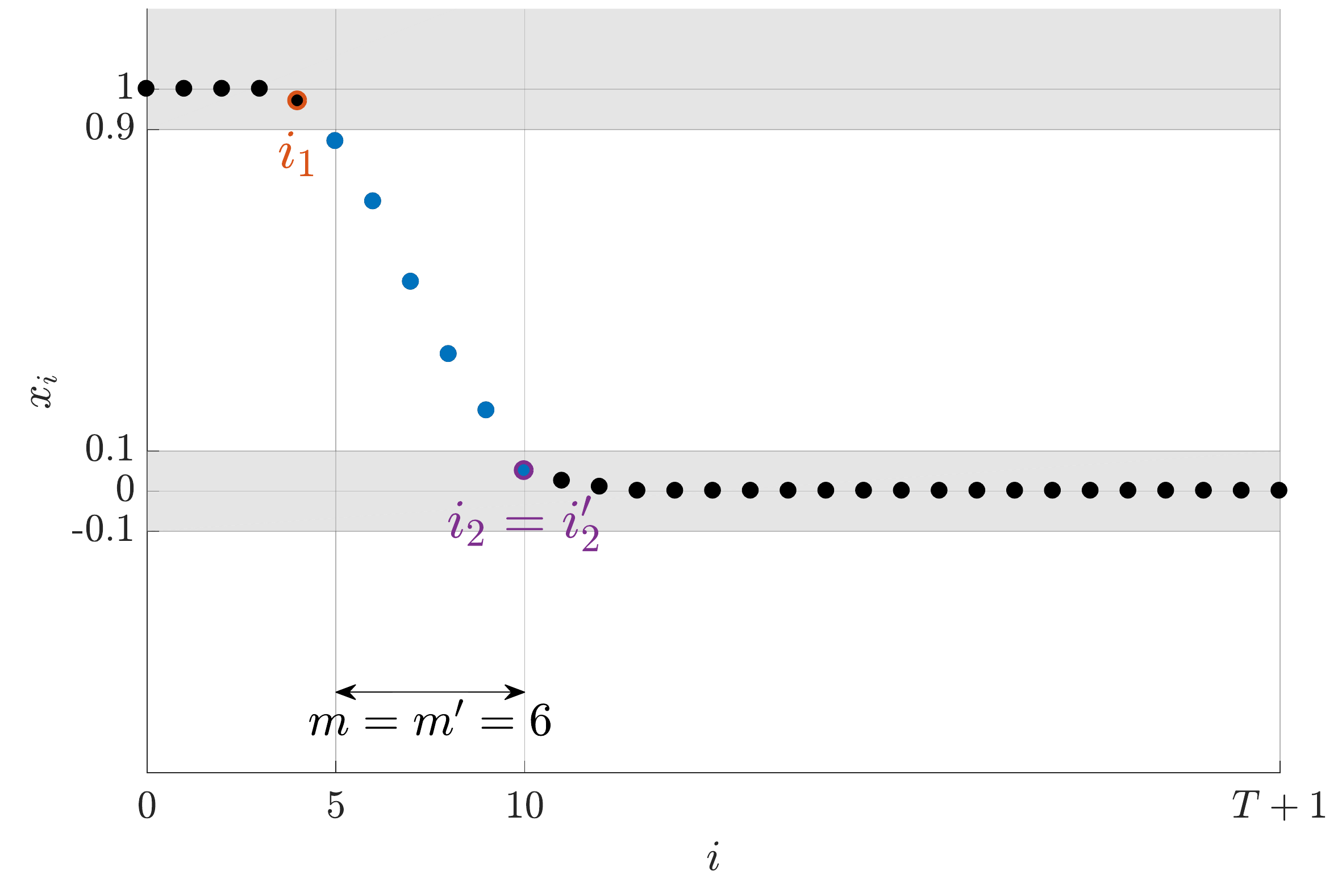} 
	\includegraphics[width=0.48\textwidth]{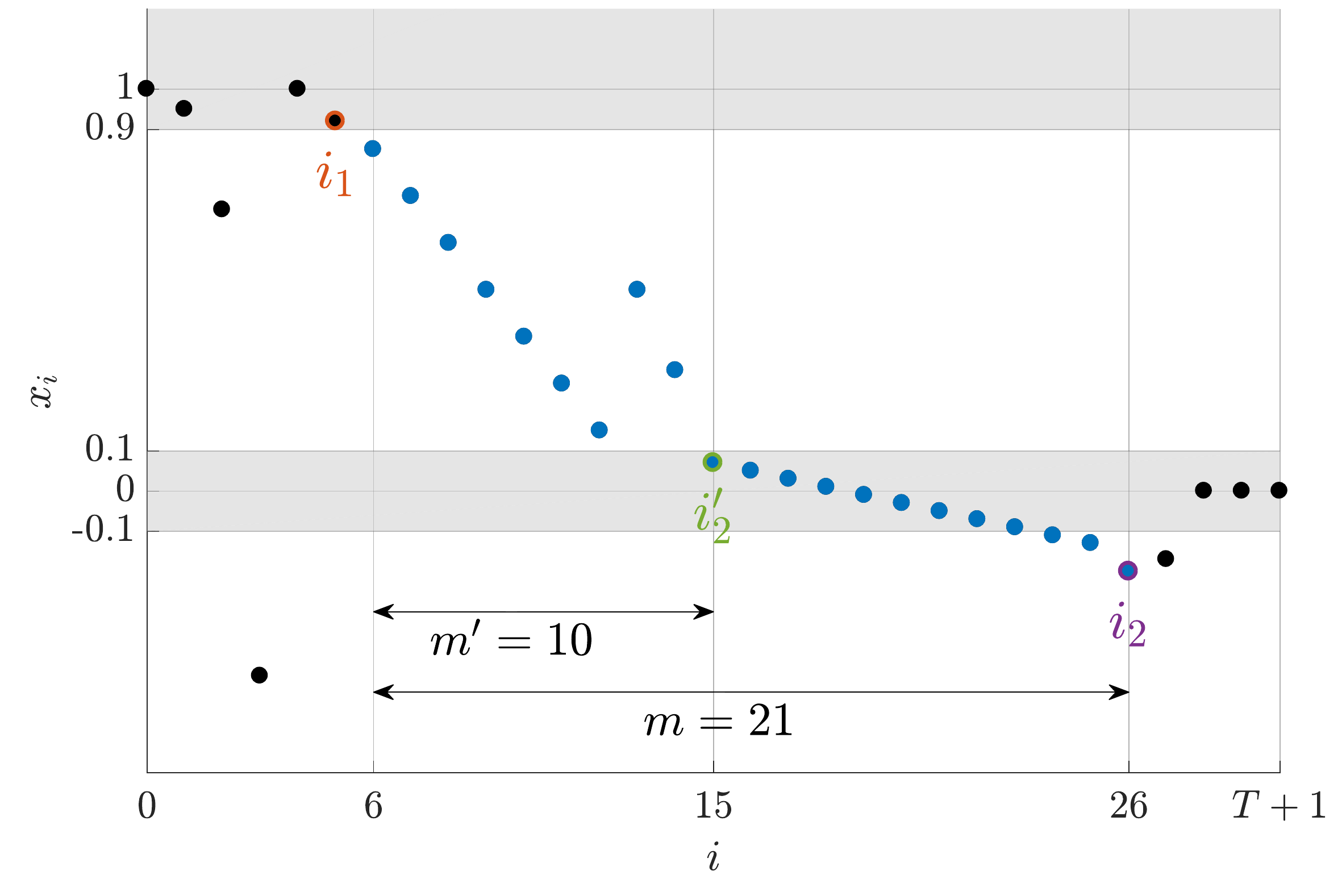} 
	\caption{\label{fig:transition-proof} The transition
		region~\eqref{eqn:trans-region} in the proof of
		Lemma~\ref{lem:gradbound}. Each plot shows the entries of a
		vector $x\in\R^{T+1}$ that satisfies $x_\T = x_{\T+1} = 0$. The
		entries of $x$ belonging to the transition region $\TransRegion$ are
		blue. }
\end{figure}
	
  Throughout the proof, we fix $x\in\R^{\T+1}$ such that $x_{T}=x_{T+1}=0$; for
  convenience in notation, we define $x_{0} \defeq 1$. Our strategy is to carefully pick two indices $\ilo \in \{0, \ldots, \T -1\}$ and $\ihi \in \{\ilo+1, \ldots, \T\}$, such that $\norm{\grad \fhardFO(x)}^2 \ge 
   \sum_{i=\ilo+1}^{\ihi} \left| \grad_i \fhardFO(x) \right|^2 > (\mu^{3/4}/4)^2$.
  We call the set of indices from $\ilo+1$ to $\ihi$ the \emph{transition region}, and construct it as follows.
  \begin{equation*}
    \mbox{Let $\ilo\ge0$ be the largest $i$ such that $x_{i}>0.9$},
  \end{equation*} 
  so that $x_{j}\le0.9$ for every $j>i$. Note that $\ilo=0$ when $x_i \le 0.9$ for every $i\in[\T+1]$. This is a somewhat special case due to the coefficient $\sqrt{\mu}\le 1$ of the first ``link'' in the quadratic chain term in~\eqref{eq:fhard-def}. To handle it cleanly we define
  \begin{equation*}
  \alpha \defeq 
  \begin{cases}
  1 & \ilo > 0 \\
  \sqrt{\mu} & \ilo = 0.
  \end{cases}
  \end{equation*}
  Continuing with construction of the transition region, we make the following definition.
  \begin{equation*}
    \mbox{
      Let $\imi\le\T$ be the smallest
      $j$ such that $j>\ilo$ and $x_{j}<0.1$, }
  \end{equation*}
  and let $m'=\imi-\ilo$, so $m' \ge 1$. Roughly, our transition
  region consists of the $m'$ indices $\ilo+1,\ldots,\imi$, but for technical
  reasons we attach to it the following decreasing 	`tail'.
  \begin{equation*}
    \mbox{Let $\ihi$ be the smallest $k$ such that $k\ge\imi$ and
      $x_{k+1}\ge x_{k}-\frac{0.2}{m'-1+1/\alpha}\indic{x_{k} > -0.1}$.}
  \end{equation*}
  With these definitions, $\ihi$ is well-defined and $0 \le \ilo < \ihi \le
  \T$, since $x_{T+1}-x_{T}=0$. We denote the transition region and 
  associated length by
  \begin{equation}
    \label{eqn:trans-region}
    \TransRegion
    \defeq \left\{ \ilo+1,\ldots,\ihi\right\}
    ~~ \mbox{and} ~~
    m \defeq \ihi-\ilo \ge 1.
  \end{equation}
  We illustrate our definition of the transition region in
  Figure~\ref{fig:transition-proof}.
	  
  Let us describe the transition region. In the ``head'' of the region, we
  have $0.1\le x_{i}\le0.9$ for every $i\in\left\{ \ilo+1,\ldots,\imi-1\right\}
  $; a total of $m'-1$ indices. The ``tail'' of the transition region is
  strictly decreasing, $x_{\ihi}<x_{\ihi-1}<\cdots<x_{\imi}$. Moreover, for any 
  $j \in \{\imi+1, \ldots \ihi-1\}$ such that $x_j > -0.1$, the decrease is rapid; 
  $x_j < x_{j - 1} - 0.2 / (m'-1+1/\alpha)$.
  This
  descriptions leads us to the following technical properties.
  \begin{lemma}
    \label{lemma:funny-transition}
    Let the transition region $\TransRegion$ be defined as
    above~\eqref{eqn:trans-region}.  Then
    \begin{enumerate}[i.]
    \item \label{item:tr-prop-edge}
      $x_{\ilo} > 0.9 > 0.1 > x_{\ihi}$ and
      $-x_{\ihi} + \left(m-1+\alpha^{-1}\right)\left(x_{\ihi+1}-x_{\ihi}\right) > -0.3$.
    \item \label{item:tr-prop-grad}
      $\Ups'\left(x_{i}\right)\le0$ for every
      $i\in\TransRegion$, and $\Ups'\left(x_{i}\right)<-1$ for at least
      $\left(m-\alpha^{-1}\right)/2$ indices in $\TransRegion$.
    \end{enumerate}
  \end{lemma}
  \noindent
  We defer the proof of the lemma to the end of this section,
  continuing the proof assuming it.

  We now lower bound $\norms{\nabla \fhardFO(x)}$.
  For notational convenience,
  define $g_{i}=\mu\Ups'\left(x_{i}\right)$, and recalling that
  $x_{T}=x_{T+1}=0$, we see that the norm
  of the gradient of $\fhardFO$ is
  \begin{flalign}
    \norm{\grad\fhardFO\left(x\right)}^{2}
    & =\left((1+\sqrt{\mu})x_{1}-\sqrt{\mu}-x_2+g_1\right)^{2}+\sum_{i=1}^{\T}\left(2x_{i}-x_{i-1}-x_{i+1}+g_{i}\right)^{2}
    \nonumber \\ &
    \ge \left((1+\alpha)x_{\ilo+1}-\alpha x_{\ilo}-x_{\ilo+2}+g_{\ilo+1}\right)^{2}
    + \sum_{i=\ilo+2}^{\ihi}\left(2x_{i}-x_{i-1}-x_{i+1}+g_{i}\right)^{2},
    \label{eqn:quad-form-to-lower}
  \end{flalign}
  where we made use of the notation $\alpha \defeq 1$ if $\ilo > 0$ and $\alpha \defeq \sqrt{\mu}$ if $\ilo =0$.
  We obtain a lower bound for the final sum of $m$
  squares~\eqref{eqn:quad-form-to-lower} by fixing
  $x_{\ilo}$, $x_{\ihi}$, and $g_{\ilo+1}, \ldots, g_{\ihi}$,
  then minimizing the quadratic form explicitly over the
  $m - 1$ variables $x_{\ilo+1},\ldots,x_{\ihi-1}$.
  We obtain
  \begin{align*}
    \norm{\grad\fhardFO\left(x\right)}^{2}
    & \ge \inf_{v \in \R^{m-1}
    }\Big\{ \left((1+\alpha)v_{1}-\alpha x_{\ilo}-v_{2}+g_{\ilo+1}\right)^{2}
    + \sum_{j=2}^{m-2}\left(2v_{j}-v_{j-1}-v_{j+1} + g_{\ilo + j}
    \right)^{2} \\
    & \qquad\qquad~ ~ +\left(2v_{m-1}-v_{m-2}-x_{\ihi}+g_{\ihi-1}\right)^{2}
    + \left(2x_{\ihi}-v_{m}-x_{\ihi+1}+g_{\ihi}\right)^{2}\Big\}
    \\
    & =\inf_{v\in\R^{m-1}}
    \norm{Av-b}^{2}
    =b^\top \left(I-A\left(A^{\top}A\right)^{-1}A^{\top}\right)b
    =\left(z^{\top}b\right)^{2},
  \end{align*}
  where the matrix $A$ and vector $b$ have definitions
  \[
  A=\left[\begin{array}{ccccc}
      1+\alpha & -1\\
      -1 & 2 & -1\\
      & \ddots & \ddots & \ddots\\
      &  & -1 & 2 & -1\\
      &  &  & -1 & 2\\
      &  &  &  & -1
    \end{array}\right]\in\R^{m\times\left(m-1\right)}
  ~~ \mbox{and} ~~ 
  b=\left[\begin{array}{c}
      \alpha x_{\ilo}-g_{\ilo+1}\\
      -g_{\ilo+2}\\
      \vdots\\
      -g_{\ihi-2}\\
      x_{\ihi}-g_{\ihi-1}\\
      -2x_{\ihi}+x_{\ihi+1}-g_{\ihi}
    \end{array}\right]
  \in \R^{m},
  \]
  and $z \in \R^{m}$ is a unit-norm solution to $A^{\top}z=0$. The
  vector $z \in \R^m$ with
  \begin{equation*}
    z_{j}= \frac{j-1+\frac{1}{\alpha}}{\sqrt{\sum_{i=1}^{m}(i-1+\frac{1}{\alpha})^{2}}}
  \end{equation*}
   is such a solution.
   Thus
  \begin{flalign}
    \lefteqn{\norm{\grad\fhardFO\left(x\right)}^{2}} \nonumber \\
    & \ge\frac{\left(x_{\ilo}-
      \sum_{j=1}^{m}\left(j-1+\frac{1}{\alpha}\right)\cdot g_{\ilo+j}
      + \big(m-2+\frac{1}{\alpha}\big)\cdot x_{\ihi} +\left(m-1+\frac{1}{\alpha}\right)(-2x_{\ihi}+x_{\ihi+1})\right)^{2}}{\sum_{i=1}^{m}(i-1+\frac{1}{\alpha})^{2}}
    \nonumber \\
    &~=\frac{1}{\sum_{i=1}^{m}(i-1+\frac{1}{\alpha})^{2}}
    \bigg(x_{\ilo}-x_{\ihi}+\left(m-1+\frac{1}{\alpha}\right)\left(x_{\ihi+1}-x_{\ihi}\right)
    -\sum_{j=1}^{m}\left(j-1+\frac{1}{\alpha}\right)\cdot g_{\ilo+j}\bigg)^{2}.
    \label{eq:trans-region-gradlb}
  \end{flalign}
  We now bring to bear the properties of the transition region
  Lemma~\ref{lemma:funny-transition} supplies. By
  Lemma~\ref{lemma:funny-transition}.\ref{item:tr-prop-edge},
  \begin{equation}
  x_{\ilo}-x_{\ihi}+(m-1+\alpha^{-1})\left(x_{\ihi+1}-x_{\ihi}\right)\ge0.9-0.3=\frac{3}{5},
  \label{eq:trans-region-x-lb}
  \end{equation}
  and by Lemma~\ref{lemma:funny-transition}.\ref{item:tr-prop-grad},  using 
  $1 \le \alpha^{-1} \le 1/\sqrt{\mu}$,
	\begin{equation}
	-\sum_{j=1}^{m}(j-1+\alpha^{-1}) g_{\ilo+j} \ge \mu\sum_{j=1}^{\left(m-\alpha^{-1}\right)/2}(j-1+\alpha^{-1})\ge 
	\frac{\mu}{8}\left[m^2-\frac{1}{\alpha^2}\right]_+ \ge \frac{1}{8}\left[\mu m^2 - 1\right]_+.
	\label{eq:trans-region-g-lb}
	\end{equation}
  Substituting $\sum_{i=1}^{m}\left(i-1+\alpha^{-1}\right)^{2} \le 
	\half m\left(m+1/\sqrt{\mu}\right)\left(m+2/\sqrt{\mu}\right)$  and the bounds~\eqref{eq:trans-region-x-lb} and~\eqref{eq:trans-region-g-lb} into the gradient lower bound~\eqref{eq:trans-region-gradlb}, 
  we have that
  \begin{equation*}
    \norm{ \grad\fhardFO\left(x\right)} \ge
    \mu^{3/4} \cdot \zeta(m\sqrt{\mu})
    ~~\mbox{where}~~
    \zeta(t) \defeq \sqrt{\frac{2}{t(t+1)(t+2)}}\left(\frac{3}{5}+\frac{1}{8}\left[t^2 - 1\right]_+ \right).
  \end{equation*}
  A quick computation reveals that $\inf_{t > 0}\zeta(t) \approx 0.28 > 1/4$, which gives the result.
\end{proof}

\begin{proof-of-lemma}[\ref{lemma:funny-transition}]
  We have by definition that $x_{\ilo} > 0.9$ and $x_{\ihi} \le x_{\imi} <
  0.1$. To see that 
  \begin{equation*}
  -x_{\ihi} + \left(m-1+\alpha^{-1}\right)\left(x_{\ihi+1}-x_{\ihi}\right) \ge -0.3
  \end{equation*}
  holds,
  consider the two cases that $x_{\ihi} \le -0.1$ or $x_{\ihi} > -0.1$.
  In the first case that $x_{\ihi}\le-0.1$, by definition
  $x_{\ihi+1}\ge x_{\ihi}$ so
  $-x_{\ihi} + \left(m-1+\alpha^{-1}\right)\left(x_{\ihi+1}-x_{\ihi}\right)>0.1>-0.3$.  The second case that
  $x_{\ihi}>-0.1$ is a bit more subtle.
  By definition of the sequence $x_{\ihi}, \ldots, x_{\imi}$, we have
  \begin{equation}
    -0.1 < x_{\ihi} < x_{\ihi - 1}
    - \frac{0.2}{m'-1+\frac{1}{\alpha}} < \cdots \le x_{\imi}
    - \frac{0.2}{m'-1+\frac{1}{\alpha}}(\ihi - \imi) < 0.1 - 0.2\frac{m-m'}{m'-1+\frac{1}{\alpha}}.
    \label{eq:trans-zero-region-inequality-chain}
  \end{equation}
  Combining this bound on $x_{\ihi}$ and the inequality $x_{\ihi + 1} \ge x_{\ihi} - \frac{0.2}{m'-1+1/\alpha}$
  due to the construction of $\ihi$,
  we obtain
  \begin{equation*}
  -x_{\ihi} + \left(m-1+\alpha^{-1}\right)\left(x_{\ihi+1}-x_{\ihi}\right) > -0.1 + 0.2\frac{m-m'}{m'-1+\frac{1}{\alpha}} - 0.2\frac{m-1+\frac{1}{\alpha}}{m'-1+\frac{1}{\alpha}} = -0.3.
  \end{equation*}
  We note for
  the proof of property~\ref{item:tr-prop-grad} that the chain of inequalities~\eqref{eq:trans-zero-region-inequality-chain} is possible only for $m \le 2m'-1+1/\alpha$, which implies there are at most $m'-1+1/\alpha$ indices
  $i \in \TransRegion$ such that $|x_i| < 0.1$.

  The first part of property~\ref{item:tr-prop-grad} follows from
  Lemma~\ref{lem:props}.\ref{item:ups-quasi-convex}, since
  $x_{i}\le 0.9 \le 1$ for every $i\in\TransRegion$.  To see that the
  second part of the property holds, let $N$ be the number of indices in
  $i\in\TransRegion$ for which $\Ups'\left(x_{i}\right)<-1$.  By
  Lemma~\ref{lem:props}.\ref{item:ups-large-grad} and the
  fact that $0.1\le x_{i}\le0.9$ for every $i\in\left\{
  \ilo+1,\ldots,\imi-1\right\} $, $N\ge m'-1$.  Moreover, since there can be
  at most $m'-1+1/\alpha$ indices $i\in\TransRegion$ for which
  $\left|x_{i}\right|<0.1$, $N\ge m-(m'-1+1/\alpha)$. Averaging the two lower bounds
  gives $N\ge\left(m-1/\alpha\right)/2$.
\end{proof-of-lemma}

\subsection{Proof of Theorem~\ref{thm:final-d}}\label{sec:D-lb-proof}

\thmFirstorderFinalDistance*

\begin{proof}
The proof builds off of those of Theorems~\ref{thm:final} and 
\exref{thm:fullder-final-dist}. We begin by recalling the following bump 
function construction
\begin{equation}
  \label{eq:bumpy-bump}
  \hhard(x) \defeq
  \compactfunc\left(1 - \frac{25}{2} \normBig{x -
    \frac{4}{5} e\ind{\T}}^2 \right)
~~\mbox{where}~~\compactfunc(t) \defeq e \cdot \exp\left(-\frac{1}{\hinge{2 t - 1}^2}\right).
\end{equation}
Adding a scaled version of $-\hhard$ to our hard instance construction allows us to ``plant'' a global minimum that is both close to the origin and essentially invisible to zero-respecting method. For convenience, we restate Lemma~\exref{lem:fullder-hhard-props},
\begin{lemma}%
  \label{lem:hhard}
  The function $\hhard$ satisfies the following.
  \begin{enumerate}[i.]
  \item \label{item:hhard-bump} For all
    $x \in \R^\T$ we have $\hhard(x) \in [0, 1]$,
    and $\hhard(0.8 e\ind{\T}) = 1$.
  \item \label{item:hhard-radius} 
    On the set $\{x\in\R^d \mid x_{\T} \le \frac{3}{5} \}
    \cup \{x \mid \norm{x} \ge 1\}$, we have $\hhard(x)=0$.
  \item \label{item:hhard-lipschitz} For every $p \ge 1$, the $p$th order
    derivative of $\hhard$ is $\smCvar{p}$-Lipschitz continuous, where
    $\smCvar{p} \le e^{c p\log p + c}$ for a numerical constant $c <
    \infty$.
  \end{enumerate}
\end{lemma}

With this lemma in place, we follow the broad outline of the proof
of Theorem~\ref{thm:final}, with modifications to make sure the norm
of the minimizers of $f$ is small. Indeed,
letting $\lambda, \sigma > 0$, we define
our scaled hard instance $f:\R^{\T + 2}\to\R$ by
\begin{equation}
  f(x) = {\lambda\sigma^2}
  \fhardFO\left(x_1 / \sigma, \ldots, x_{\T + 1} / \sigma\right)
  - {\tilde{\lambda}} \hhardPlus\left({x}/{D}\right),
  \label{eq:hard-distance-instance}
\end{equation}
that is, the hard instance we construct in Theorem~\ref{thm:final} minus a
scaled bump function~\eqref{eq:bumpy-bump}.  For every $p \in \N$, we set
the parameters $\lambda, \sigma, \mu$ and $r$ as in the proof of
Theorem~\ref{thm:final}, so that we satisfy inequality~\eqref{eq:lb-lq}
except we replace $\Sm{q}$ with $\Sm{q}/2$ for every $q\in[p]$ (including in
the definitions of  $\lambda, \sigma, \mu$). Thus, as in
inequality~\eqref{eq:lb-lq}, for each $q \in \N$ the function $f_0(x) \defeq
\lambda\sigma^2\fhardFO\left(x/\sigma\right)$ has $\Sm{q}/2$-Lipschitz $q$th
order derivative and satisfies $\norm{\grad f_0(x)}>\epsilon$ for all
$x\in\R^{\T+1}$ with $x_{\T} = x_{\T+1} = 0$. By
Lemma~\ref{lem:hhard}.\ref{item:hhard-lipschitz},
setting
\begin{equation}
  \tilde{\lambda} = \min_{q\in[p]} \frac{1}{2\smCvar{q}}\Sm{q}D^{q+1}
  \label{eq:lb-d-lambda-tilde}
\end{equation}
guarantees that the function $x \mapsto - \tilde{\lambda}\cdot
\hhardPlus(x/D)$ also has $\Sm{q}/2$-Lipschitz $q$th order derivatives, so
that overall, for each $q \in [p]$ the function $f$ defined in
Eq.~\eqref{eq:hard-distance-instance} has $\Sm{q}$-Lipschitz $q$th order
derivative. 

We note that by Lemma~\ref{lem:hhard}.\ref{item:hhard-radius},
$\hhardPlus(x)$ is identically 0 at a neighborhood of any $x$ with
$x_{\T+2}=0$, which immediate implies that $\hhardPlus$ and $f$ are
zero-chains. Therefore for any $\alg\in\AlgZRFO$ producing iterates
$x\ind{1}=0,x\ind{2},x\ind{3},\ldots$ when operating on $f$, we have 
$x\ind{t}_\T =
x\ind{t}_{\T+1}=x\ind{t}_{\T+2}=0$ for any $t\le\T$. Thus, by our choices of
$\lambda,\sigma,\mu$ and $r$, $\norm{\grad f(x\ind{t})} = \norm{\grad
  f_0(x\ind{t})} > \epsilon$ for every $t\le\T$, and so
\begin{equation*}
  \TimeEps{\AlgZRFO}{\FclassDMany{p}} \ge \inf_{\alg\in\AlgZRFO}
  \TimeEps {\alg}{f} \ge T+1. 
\end{equation*}

To establish that $f\in\FclassDMany{p}$, it remains to show that every global 
minimizer of $f$ has norm at most $D$. Let $x^\star$ denote a global
minimizer of $f$, and temporarily assume that
\begin{equation}\label{eq:distance-assumption}
	f\left( 0.8D\cdot e\ind{\T+2}\right) < 0,
\end{equation}
Therefore, $f(x^\star)< f\left( 0.8D\cdot e\ind{\T+2}\right)<0$ and
$\hhardPlus(x^\star/D) \ne 0$, as otherwise we have the contradiction
$f(x^\star) = \lambda\sigma^2 \fhardFO(x^\star/\sigma) \ge 0$. By the
definition~\eqref{eq:bumpy-bump}, $\hhardPlus(x^\star/D) \ne 0$ implies that
$1-\frac{25}{2}\norm{x^\star/D - 0.8e\ind{T+2}}^2\ge0.5$, and therefore
$\norm{x^\star}\le D$. To verify the assumed
inequality~\eqref{eq:distance-assumption}, we use
Lemma~\ref{lem:hhard}.\ref{item:hhard-bump} to obtain
\begin{equation*}
  f\left( 0.8D \cdot e\ind{\T+2}\right) = {\lambda\sigma^2}\cdot \fhardFO\left(0\right) - {\tilde{\lambda}}\cdot \hhardPlus\left(0.8\cdot e\ind{T+2}\right) = \frac{\lambda\sqrt{\mu}\sigma^2}{2} + 10\lambda\sigma^2\mu\T - {\tilde{\lambda}}.
\end{equation*}
Therefore, if we set
\begin{equation}\label{eq:lb-d-genlb}
  T = \floor{\frac{\tilde{\lambda}-\lambda\sqrt{\mu}\sigma^2/2}{10\lambda\mu\sigma^2}}
\end{equation}
then inequality~\eqref{eq:distance-assumption} holds and $\norm{x^\star}\le
D$, and so $f\in\FclassDMany{p}$.  Comparing the
setting~\eqref{eq:lb-d-genlb} of $T$ above to the setting~\eqref{eq:lb-T} of
$T$ in the proof of Theorem~\ref{thm:final}, we see they are identical
except that we replace the term $\DeltaF$ in~\eqref{eq:lb-T} with
$\tilde{\lambda} \defeq \min_{q\in[p]} ({2\smCvar{q})^{-1}}\Sm{q}D^{q+1}$.
Thus, mimicking the proof of Theorem~\ref{thm:final} after
the step~\eqref{eq:lb-T}, \emph{mutatis mutandis}, yields the result.
\end{proof}

\subsection{Proof of Lemma~\ref{lem:disc-tight}}
\label{sec:proof-disc-tight}

\lemDistTight* 

\begin{proof}
We construct $x$ as follows. We let $x_1 = 1$, and for $n>1$ let (with $x_0 
\defeq 1$),
\begin{equation*}
  x_n = x_{n-1} - (x_{n-2}-x_{n-1}) - \delta_{n-1}   = 1 - \sum_{i=1}^{n-1} \sum_{j=1}^i \delta_i \,,
\end{equation*}
where we take 
\begin{equation*}
  \delta_n = \frac{1}{m(m+1)}
  \begin{cases}
    1 & n \le m \\
    0 & n=m+1 \text{ or } n > 2m +1 \\
    -1 & m+1 < n \le 2m +1
  \end{cases}
\end{equation*}
for some $m\in \N$ which we will later determine. The elements of $\grad \fhardTFO$ are given by
\begin{equation*}
  \grad_n \fhardTFO(x) = \linkfun'(x_n -x_{n-1}) - \linkfun'(x_{n+1} -x_n ) + \mu \bupsilon'(x_n),
\end{equation*}
where for $n=1$ we used $x_1=1$ and $\linkfun'(0)=0$ to write $\alpha \cdot \linkfun'(x_1-1) = 0 = \linkfun'(x_1-1)$. Since $\linkfun'$ is 1-Lipschitz, we have
\begin{equation*}
  \left| \linkfun'(x_n -x_{n-1}) - \linkfun'(x_{n+1} -x_n ) \right| \le 
  \left| (x_n -x_{n-1}) - (x_{n+1} -x_n )\right| = \left| \delta_n \right|.
\end{equation*}
Moreover, one can readily verify that $x_n\in [0,1]$ for every $n$ and that $x_n = 0$ for every $n > 2m+1$. Therefore, using using $\bupsilon'(0) = 0$ and $\max_{z\in [0,1]} | \bupsilon'(z) | \le G$ we have that $\left|\bupsilon'(x_n)\right| \le G\cdot \indic{n \le 2m+1}$, which gives the overall bound
\begin{equation*}
  \left|\grad_n \fhardTFO(x)\right| \le 
  \left| \delta_n \right| + \mu\left|\bupsilon(x_n)\right| \le
  \left( \frac{1}{m^2} + G\mu\right)\indic{n \le 2m+1},
\end{equation*}
and thus,
\begin{equation*}
  \norm{\grad \fhardTFO(x)} \le \sqrt{2m+1}\left(m^{-2} + G\mu\right) \le \sqrt{3}\left(m^{-3/2} + \sqrt{m}G\mu\right).
\end{equation*}
Taking $m = \ceil{\frac{1}{3\sqrt{\mu}}}$, we have
\begin{equation*}
  \norm{\grad \fhardTFO(x)} \le \sqrt{3}\left( \ceil{\frac{1}{3\sqrt{\mu}}}^{-3/2} + G\ceil{\frac{1}{3\sqrt{\mu}}}^{1/2}\mu \right) \le \left( 27 + \sqrt{3}G\right) \mu^{3/4},
\end{equation*}
where we have used $\ceil{1/(3\sqrt{\mu})} \le 1/\sqrt{\mu}$ since $\mu \le 
1$. Thus, $ \norm{\grad \fhardTFO(x)} \le C\mu^{3/4}$ holds for $C=27 + 
\sqrt{3}G$. For $T \ge 8$, since $\mu \ge \T^{-2}$, we have $2m+1 \le 
2\ceil{\T/3}+1 < \T$ and therefore  $x_\T = x_{\T+1}=0$ holds as required 
(since $x_n = 0$ for every $n > 2m+1$). In the edge case $\T \le 8$ we have 
$\mu \ge \T^{-2} \ge 1/64$ and therefore  $x=0$ yields $\norm{\grad 
\fhardTFO(x)} = \alpha \le 1 \le 27\cdot(1/64)^{3/4} \le C\mu^{3/4}$.
\end{proof}

\end{document}